\documentclass[12pt]{article}
\usepackage{amsthm, amssymb, bm}
\usepackage{soul, color}
\usepackage[]{amsmath}
\usepackage[]{amsfonts}
\usepackage[]{fancyhdr}
\usepackage[]{graphicx}
\usepackage[section]{algorithm}
\usepackage{algorithmic}
\graphicspath{{/EPSF/}{../figures/}{figures/}}

\newtheorem{theorem}{Theorem}[section]
\newtheorem{lemma}[theorem]{Lemma}
\newtheorem{proposition}[theorem]{Proposition}

\newtheorem{remark}[theorem]{Remark}

\def \Cm {\mathbb{C}}
\def \Imm {\mathbb{I}}

\def \Rm {\mathbb{R}}
\def \Sm {\mathbb{S}}

\def\A{\mathcal{A}}
\def\C{\mathcal{C}}

\def\O{\mathcal{O}}
\def\P{\mathcal{P}}

\newcommand{\eps}{\varepsilon}

\newcommand{\dint}{\displaystyle\int}
\newcommand{\pdr}[2]{\dfrac{\partial{#1}}{\partial{#2}}}

\renewcommand{\div}{\nabla\cdot} 
 
\newcommand{\curl}{\nabla\times}

\newcommand{\wtD}{ {\widetilde D} }

\newcommand{\wtH}{ {\widetilde H} }
\newcommand{\wtM}{ {\widetilde M} }

\newcommand{\wtQ}{ {\widetilde Q} }
\newcommand{\wtR}{ {\widetilde R} }
\newcommand{\wtS}{ {\widetilde S} }
\newcommand{\wtT}{ {\widetilde T} }
\newcommand{\wtV}{ {\widetilde V} }

\newcommand{\qandq}{\quad\text{ and }\quad}
\newcommand{\bfe}{ {\bf e}}
\newcommand{\bfg}{ {\bf g}}

\newcommand{\bfk}{ {\bf k}}

\newcommand{\bfR}{ {\bf R}}

\newcommand{\cout}[1]{}

\newcommand{\brho}{\boldsymbol \rho}

\newcommand{\fc}{{\mathfrak c}}

\title{Inverse diffusion from knowledge of power densities}
\author{Guillaume Bal\thanks{Department of Applied Physics and 
 Applied Mathematics, Columbia University,  New York NY, 10027; gb2030@columbia.edu},\;\;Eric Bonnetier\thanks{Laboratoire Jean Kuntzmann, Universit\'e 
de Joseph Fourier \& CNRS, 38041 Grenoble Cedex 9, France, E-mail:  Eric.Bonnetier@imag.fr},\; Fran\c cois Monard\thanks{Department of Applied Physics and Applied Mathematics, Columbia University, 
        New York NY, 10027; fm2234@columbia.edu},\; Faouzi Triki\thanks{Laboratoire Jean Kuntzmann, Universit\'e 
de Joseph Fourier \& CNRS,
 38041 Grenoble Cedex 9, France,  
E-mail:  Faouzi.Triki@imag.fr}}

\begin{document}
\maketitle
\begin{abstract}
This paper concerns the reconstruction of a diffusion coefficient in an elliptic equation from knowledge of several power densities. The power density is the product of the diffusion coefficient with the square of the modulus of the gradient of the elliptic solution. The derivation of such internal functionals  comes from perturbing the medium of interest by acoustic  (plane) waves, which results in small changes in the diffusion coefficient. After appropriate asymptotic expansions and (Fourier) transformation, this allow us to construct the power density of the equation point-wise inside the domain. Such a setting finds applications in ultrasound modulated electrical impedance tomography and ultrasound modulated optical tomography. 

We show that the diffusion coefficient can be uniquely and stably reconstructed from knowledge of a sufficient large number of power densities. Explicit expressions for the reconstruction of the diffusion coefficient are also provided. Such results hold for a large class of boundary conditions for the elliptic equation in the two-dimensional setting. In three dimensions, the results are proved for a more restrictive class of boundary conditions constructed by means of complex geometrical optics solutions.
\end{abstract}

%\tableofcontents

%
%%
%%%%%%%%%%%%%%
\section{Introduction}
\label{sec:intro}

Optical tomography (OT) and electrical impedance tomography (EIT) are medical imaging techniques that utilize the large contrast between the optical and electrical response of certain unhealthy tissues and that of healthy tissues. However, because OT and EIT are modeled by diffusion operators that are highly smoothing, the reconstructions are typically very low resolution \cite{AS-IP-10,B-IP-09,U-IP-09}. Several recent imaging techniques aim to combine the high contrast in OT and EIT with a high resolution modality, for instance based on magnetic resonances (MREIT)  \cite{KWYS-IEEE-02,NTT-IP-07,NTT-IP-09} or, for what is of interest for us here, ultrasound. There are several ways of combining EIT and OT with ultrasound. One way is to use the photo-acoustic or the electro-acoustic effects; see, e.g., \cite{BU-IP-10,CAB-JOSA-09,PS-IP-07,W-JDM-04}. In this paper, we consider another way based on ultrasound modulation. 

In the setting of EIT, ultrasound modulated electrical impedance tomography (UMEIT) is also called acousto-electric tomography; see \cite{ABCTF-SIAP-08,B-UMEIT-11,CFGK-SJIS-09,GS-SIAP-09,KK-AET-11,ZW-SPIE-04}. In the setting of optical tomography, ultrasound modulated tomography (UMOT) is also called acousto-optic tomography. A model has been derived in \cite{BS-PRL-10} in the so-called incoherent regime of wave propagation, while the coherent regime is addressed in, e.g., \cite{AFRBG-OL-05,KLZZ-JOSA-97,W-JDM-04}.

\medskip

The inverse problem of UMEIT may be described as follows. The electric potential $u$ solves a diffusion equation (see \eqref{eq:ellipticDsigma} below) with unknown conductivity $\sigma$. The ultrasound modulation whose modeling is recalled in section \eqref{sec:derivation} provides access to the power density $H(x)=\sigma(x)|\nabla u|^2(x)$ for all $x$ inside the domain of interest. Reconstruction of $\sigma$ from this one internal functional $H(x)$ is typically not feasible unless additional constraints are imposed; see e.g., \cite{B-UMEIT-11}. In this paper, we assume that several such power densities corresponding to, e.g., different boundary conditions in \eqref{eq:ellipticDsigma} are known. By polarization or direct measurements as shown in section \ref{sec:derivation}, we can have access to $H_{ij}(x)=\sigma(x)\nabla u_i\cdot\nabla u_j(x)$ for $u_i$ and $u_j$ solutions to the same elliptic equation with different boundary conditions. The case of the spatial dimension $n=2$ was treated in \cite{CFGK-SJIS-09}. This paper aims to generalize these results to the case $n=3$ (which also generalize to the case $n\geq 4$, see \cite{MB}, although we shall not present the details here) and to show that UMEIT is a stable inverse problem, in the sense that in an appropriate norm, errors on the measurement of the power densities $H_{ij}$ result in errors on the reconstruction of $\sigma$ that are of the same order. This should be contrasted with the case of EIT and OT, where the error on $\sigma$ is roughly proportional to the logarithm of (and hence much larger than) the error on the available (Cauchy) data. Moreover, our derivation is constructive in the sense that it produces an explicit procedure to reconstruct $\sigma$ from knowledge of a sufficiently large number of internal measurements of the form $H_{ij}$.

\medskip

In dimension $n=2$, one set of measurements for $1\leq i,j\leq 2$ is sufficient to uniquely characterize $\sigma$. In dimension $n=3$, sets of measurements $1\leq i,j\leq 3$ are sufficient locally but not necessarily globally. We show the existence of sets of measurements $1\leq i,j\leq 4$, which uniquely (and stably) characterize $\sigma$. This construction is based on showing that the determinant of a matrix whose columns are the gradients of three of the four solutions to the elliptic equation is always positive at any point $x$ of the domain. The three chosen solutions depend on the point $x$. Such constructions are based on the use of complex geometric optics solutions, which are a convenient tool to show that solutions to elliptic equations satisfy desired qualitative properties \cite{B-UMEIT-11,BRUZ-IP-11,BU-IP-10,T-IP-10}.

\medskip

The rest of the paper is structured as follows. The derivation of the power density measurements from ultrasound modulations of media of interest is recalled in section \ref{sec:derivation}. The main results obtained in this paper are presented in section \ref{sec:main}. The proofs of the main theorems are given in the following two sections. Section \ref{sec:local} provides the proof of uniqueness, stability, and reconstruction procedures on a domain $\Omega$ on which the determinant of a matrix composed of $n$ gradients of solutions is shown to be bounded below by a positive constant. How such results can be patched together to provide global reconstruction results is given in section \ref{sec:global}. That determinants can be shown to have a prescribed sign is based on the use of complex geometric optics (CGO) solutions that are presented in section \ref{sec:global}.

%
%%
%%%%%%%%%%%%%%%%%%%
\section{Derivation of internal functionals}
\label{sec:derivation}

We consider the following elliptic equation
\begin{equation}
  \label{eq:ellipticDsigma}
  -\nabla\cdot \sigma(x)\nabla u = 0\quad \mbox{ in } X, \qquad u=g\quad \mbox{ on } \partial X.
\end{equation}
Here, $\sigma$ is the unknown diffusion coefficient, which we assume in this paper is a real-valued, scalar, function defined on a simply connected domain $X\subset\Rm^n$ for $n=2$ or $n=3$. We assume that $\sigma$ is bounded above and below by positive constants so that the above equation admits a unique solution. We also assume that $\sigma$ is sufficiently smooth so that the solution to the above equation is continuously differentiable on $\bar X$, the closure of $X$ \cite{gt1}. We denote by $\partial X$ the boundary of $X$ and by $g(x)$ the imposed (sufficiently smooth) Dirichlet boundary conditions. %All the results derived in this paper also hold if the latter conditions are replaced by Neumann or Robin boundary conditions.

The coefficient $\sigma(x)$ may model the electrical conductivity in the setting of electrical impedance tomography (EIT) or a diffusion coefficient of particles (photons) in the setting of optical tomography (OT). Both EIT and OT are modalities with high contrast, in the sense that $\sigma(x)$ takes different values in different tissues and allows one to discriminate between healthy and non-healthy tissues. In OT, high contrasts are mostly observed in the absorption coefficient, which is not modeled here \cite{BS-PRL-10}.

A methodology to couple high contrast with high resolution consists of perturbing the diffusion coefficient acoustically. Let an acoustic signal propagate through the domain. In this presentation, we assume that the sound speed is constant and that the acoustic signal is a plane wave of the form $p\cos(k\cdot x + \varphi)$ where $p$ is the amplitude of the acoustic signal, $k$ its wave-number  and $\varphi$ an additional phase. The acoustic signal modifies the properties of the diffusion equation. We assume that such an effect is small and that the coefficient in \eqref{eq:ellipticDsigma} is modified as 
\begin{equation}
  \label{eq:modifcoefs}
  \sigma_\eps(x) = \sigma(x) (1+\zeta \eps \fc ),
\end{equation}
where we have defined $\fc=\fc(x)=\cos(k\cdot x + \varphi)$ and where $\eps=p\Gamma$ is the product of the acoustic amplitude $p\in\Rm$ and a measure $\Gamma>0$ of the coupling between the acoustic signal and the modulations of the constitutive parameter in \eqref{eq:ellipticDsigma}, see \cite{KK-AET-11} and references therein. We assume that $\eps\ll1$ so that the influence of the acoustic signal on $\sigma_\eps$ admits an asymptotic expansion that we truncated at the second order as displayed in \eqref{eq:modifcoefs}. The size of the terms in the expansion are physically characterized by $\zeta$ and depend on the specific application.

Let $u$ and $v$ be solutions of \eqref{eq:ellipticDsigma} with fixed boundary conditions $g$ and $h$, respectively. When the acoustic field is turned on, the coefficients are modified as described in \eqref{eq:modifcoefs} and we denote by $u_\eps$ and $v_\eps$ the corresponding solutions. Note that $u_{-\eps}$ is the solution obtained by changing the sign of $p$ or equivalently by replacing $\varphi$ by $\varphi+\pi$.

By the standard continuity of the solution to \eqref{eq:ellipticDsigma} with respect to changes in the coefficients and regular perturbation arguments, we find that $u_\eps=u_0+\eps u_{1} + O(\eps^2)$.
% and that $u_\eps u_{-\eps}-u_0^2=O(\eps^2)$
Let us multiply the equation for $u_\eps$ by $v_{-\eps}$ and the equation for $v_{-\eps}$ by $u_\eps$, subtract the resulting equalities, and use standard integrations by parts. We obtain that 
\begin{equation}
  \label{eq:Green}
  \dint_X (\sigma_\eps-\sigma_{-\eps})\nabla u_\eps\cdot\nabla v_{-\eps} dx = \dint_{\partial X}  \sigma_{-\eps} \pdr{v_{-\eps}}{\nu} u_\eps -\sigma_\eps \pdr{u_\eps}{\nu} v_{-\eps}  ds(x).
\end{equation}
Here, $ds(x)$ is the standard surface measure on $\partial X$.
We assume that $\sigma_\eps \partial_\nu u_\eps$ and $\sigma_\eps \partial_\nu v_\eps$ are measured on $\partial X$, at least on the support of $v_{\eps}=h$ and $u_\eps=g$, respectively, for all values $\eps$ of interest. Note that the above equation holds if the Dirichlet boundary conditions are replaced by Neumann boundary conditions. Let us define
\begin{equation}
  \label{eq:measbdry}
  J_\eps :=  \dfrac12\dint_{\partial X}  \sigma_{-\eps} \pdr{v_{-\eps}}{\nu} u_\eps -\sigma_\eps \pdr{u_\eps}{\nu} v_{-\eps}  ds(x)  \,\,=\,\, \eps J_1 + \eps^2 J_2 + O(\eps^3).
\end{equation}
We assume that the real valued functions $J_m=J_m(k,\varphi)$ are known (measured functions). Notice that such knowledge is based on the physical boundary measurement of the Cauchy data of the form $(u_\eps,\sigma_\eps \partial_\nu u_\eps)$ and  $(v_\eps,\sigma_\eps \partial_\nu v_\eps)$ on $\partial X$. 

Equating like powers of $\eps$, we find that at leading order
\begin{equation}
  \label{eq:leadingorder}
  \dint_X \big[\zeta \sigma(x) \nabla u_0\cdot\nabla v_0(x) \big] \cos(k\cdot x+\varphi) dx = J_1(k,\varphi).
\end{equation}
This may be acquired for all $k\in\Rm^n$ and $\varphi=0,\frac\pi2$, and hence provides the Fourier transform of
\begin{equation}
  \label{eq:H0Lap}
  H[u_0,v_0](x) = \zeta \sigma(x) \nabla u_0\cdot\nabla v_0(x).
\end{equation}

Note that when $v_\eps=u_\eps$, then we find from the expression in \eqref{eq:Green}  that $J_2=0$ in \eqref{eq:measbdry} so that the expression for $J_1$ may be obtained from available measurements in \eqref{eq:measbdry} with an accuracy of order $O(\eps^2)$. Note also that 
\begin{displaymath} 
  H[u_0,v_0](x) = \dfrac14\big(H[u_0+v_0,u_0+v_0]-H[u_0-v_0,u_0-v_0]\big)
 \end{displaymath}
by polarization. In other words, the limiting measurements (for small $\eps$) in \eqref{eq:H0Lap} may also be obtained by considering expressions of the form $\eqref{eq:Green}$ with $u_\eps=v_\eps$.

In the setting of optical tomography, the coefficient $\sigma_\eps$ in \eqref{eq:modifcoefs} takes the form 
\begin{displaymath} 
   \sigma_\eps(x) = \dfrac{\tilde \sigma_\eps}{c_\eps^{n-1}}(x),
 \end{displaymath}
where $\tilde \sigma_\eps$ is the diffusion coefficient, $c_\eps$ is the light speed, and $n$ is spatial dimension. When the pressure field is turned on, the location of the scatterers is modified by compression and dilation. Since the diffusion coefficient is inversely proportional to the scattering coefficient, we find that 
\begin{displaymath} 
   %\tilde\sigma_\eps(x) = \tilde\sigma(x) \big( 1+\eps \fc(x)\big),\qquad
 \dfrac1{\sigma_\eps(x)}=\dfrac1{\sigma(x)}\big(1+\eps \fc(x)\big).
 \end{displaymath}
Moreover, the pressure field changes the index of refraction (the speed) of light as follows
\begin{displaymath} 
    c_\eps(x) = c(x)(1+\gamma\eps \fc(x)),
\end{displaymath}
where $\gamma$ is a constant (roughly equal to $\frac13$ for water). This shows that
\begin{equation}
  \label{eq:zetaeta}
  \zeta = - (1+(d-1)\gamma).
\end{equation}
In the setting of electrical impedance tomography, we simply assume that $\zeta$ models the coupling between the acoustic signal and the change in the electrical conductivity of the underlying material.
The value of $\zeta$ thus depends on the application.  

The objective of ultrasound modulated optical tomography (UMOT) and ultrasound modulated electrical impedance tomography (UMEIT) is to reconstruct the coefficient $\sigma(x)$ from measurements of the form \eqref{eq:H0Lap}, i.e., since we assume that $\zeta$ is known, from measurements of the form $H(x) =\sigma(x) \nabla u(x) \cdot \nabla v(x)$, where $u$ and $v$ are two solutions of the unperturbed equation with (possibly) different boundary conditions.

In the setting where $v_0=u_0$, measurements are of the form $H_{00}(x) = \sigma(x)|\nabla u_0|^2$. Plugging the latter expression into the elliptic equation yields the nonlinear equation
\begin{equation}
  \label{eq:0Lap}
  \nabla\cdot \dfrac{H_{00}(x)}{|\nabla u_0|^2} \nabla u_0 =0 \quad\mbox{ in } X, \qquad u_0=g \quad\mbox{ on } \partial X.
\end{equation}
We thus observe that the reconstruction of $\sigma$ may be recast as solving the above nonlinear partial differential equation. It turns out that the above equation is not directly amenable to analysis. %modified versions of the above inverse problem will be considered in section \%\%\cout{\ref{sec:diff}}. 
However, the above expression makes clear the connection between inverse problems with internal information and nonlinear partial differential equations.

The methodology to obtain \eqref{eq:H0Lap} follows the presentation in \cite{BS-PRL-10} and is very similar in spirit to the derivation obtained in \cite{KK-AET-11}. An alternative method based on physical focusing in time of acoustic signals has been presented in \cite{ABCTF-SIAP-08}.

Assuming the validity of the above derivation, the mathematical problem of interest in this paper is as follows. We want to reconstruct $\sigma(x)$ from knowledge of the interior functionals
\begin{align}
    H_{ij} (x) = \sigma(x) \nabla u_i (x) \cdot\nabla u_j (x), \qquad 1\le i,j\le m,
    \label{eq:Hij}
\end{align} 
where $u_j$ is the solution to the equation 
\begin{align}
    \begin{split}
	\nabla\cdot(\sigma\nabla u_i) &= 0 \quad X, \\
	u_i &= g_i \quad \partial X,	\qquad 1\leq i\leq m,
    \end{split}
    \label{eq:conductivity}
\end{align}
for appropriate choices of the boundary conditions $g_i$ on $\partial X$. Obviously, we would like $m$ to be as small as possible and ideally equal to $1$. Our results propose explicit reconstructions for $m=n$ in dimension $n=2$ and $m=n$ or $m=n+1$ depending on additional parameters in dimension $n=3$. The case of the dimension $n\geq4$ can be handled in a similar fashion although we shall not present the mathematically more involved and practically less interesting details in this paper.

%
%
%
%
%%%%%%%%%%%%%%%%%%%%%%%%%%
\section{Statement of the main results} \label{sec:main}

The main strategy we use to reconstruct $\sigma$ from knowledge of $H=\{H_{ij}\}_{ij}$ is first to write equations for the quantities $S_i := \sqrt{\sigma} \nabla u_i$ that are independent of $\sigma$ and then to show that $\sigma$ is uniquely determined when $S_i$ is known. This was implemented in the two dimensional setting in \cite{CFGK-SJIS-09} and the main result of this paper is the extension to the three dimensional case. The extension to arbitrary dimensions is also feasible although we shall restrict ourselves to the cases $n=2,3$ for concreteness.

We define:
\begin{align}
   S_i := \sqrt\sigma \nabla u_i, \quad 1\leq i\leq m ,\qquad  F:= \nabla(\log\sqrt{\sigma}) = \frac{1}{2} \nabla \log \sigma.
    \label{eq:defF}
\end{align}
Using the equations \eqref{eq:conductivity}, the definitions \eqref{eq:defF} and the fact that $\sigma^{-\frac12}S_i$ is a gradient, we obtain:
\begin{align}
    \nabla \cdot (\sqrt{\sigma} S_j ) = 0 \quad&\Leftrightarrow\quad \nabla\cdot S_j + F\cdot S_j = 0, \label{eq:divSj} \\[2mm]
    n=2\,:\quad  \left[\nabla,\frac{1}{\sqrt{\sigma}} S_j \right] = 0 \quad&\Leftrightarrow\quad [\nabla, S_j] - [F,S_j] =0, \label{eq:curlSj2d}  \\[2mm]
    n=3\, :\quad \nabla\times \left(\frac{1}{\sqrt{\sigma}} S_j \right) = 0 \quad&\Leftrightarrow\quad \nabla\times S_j - F\times S_j = 0, \label{eq:curlSj3d}
\end{align}
where we have defined for $n=2$ the product $[A,B] := A_x B_y - A_y B_x$ and $[\nabla, A] := \partial_x A_y - \partial_y A_x$ for smooth vector fields $A$ and $B$, while for $n=3$, $\times$ is the standard cross product.

We now wish to eliminate $F$ from such equations and get a closed form equation for the vectors $S_i$ with sources that only involve the known matrix $H$. Such an elimination requires that we find $n$ vectors $S_i$ that form a basis of $\Rm^n$ for $n=2,3$. In dimension $n=2$, it is not difficult to find such a basis for all $x\in X$. In dimension $n\geq3$, we are able to construct such bases on subset of $\Rm^n$ that cover $X$ using triplets of vectors that depend on the subset. We thus consider the above equations for $S_i$ and $F$ over an open subset $\Omega\subset X$, over which we make the assumption that 
\begin{align}\label{eq:positiv}
    \inf_{x\in\Omega} \det (S_1(x),\dots,S_n(x)) = c_0 >0,
\end{align}
where the $n$ vectors $S_j$ for $1\leq j\leq n$ are chosen among the $m$ vectors considered in \eqref{eq:Hij}. Since the determinants are central in our derivations, we define
\begin{equation}
  \label{eq:dD}
  d(x) := \det (S_1(x),S_2(x)), \,\,n=2\quad \mbox{ and } \quad
  D(x) := \det (S_1(x),S_2(x),S_3(x)),\,\,n=3,
\end{equation}
for all $x\in \Omega$. Note that $\det (S_1(x),\dots,S_n(x))=({\rm det} H )^{\frac12}(x)\geq c_0$ on $\Omega$.
%%%%%-------------------------------------------------------------------------------------------------------
\begin{remark}  \label{rem:Hmun} Since $H$ is invertible and is a symmetric non-negative matrix, it satisfies the following inequalities 
 \begin{equation*}
     \|H^{-1}\|_\infty^{-n} \le \det H \le \prod_{i=1}^n H_{ii} \le \|H\|_\infty^n,
\end{equation*}
where the second inequality is a general property of gramian matrices (i.e. matrices of dotproducts of a given family of vectors). Indeed, for $n=2$, we have $\det H = H_{11}H_{22}- H_{12}^2 \le H_{11} H_{22}$, and for $n=3$, we may write
\begin{align*}
    \det H = (\det(S_1,S_2,S_3))^2 = \| S_1\cdot (S_2\times S_3) \|^2 \le \|S_1\|^2 \|S_2\times S_3\|^2 \le \|S_1\|^2 \|S_2\|^2 \|S_3\|^2,
\end{align*}
after successively using the Cauchy-Schwarz inequality and a property of the cross-product. Then, since $H = S^T S \in C^0(\overline\Omega)$ (by regularity of the solutions and assuming $\sigma$ continuous), the assumption~\eqref{eq:positiv} is equivalent to $H$ being invertible over $\overline \Omega$.  In this case $H^{-1}$  is uniformly bounded over
$\Omega$,  that is $\|H^{-1}\|_\infty < +\infty$, and so  
 \begin{equation}
  \label{eq:Hmoins1}
 0< \|H^{-1}\|_\infty^{-n} \le c_0^2.
\end{equation}
\end{remark}
%%%%--------------------------------------------------------------------------------------------------------------
In two dimensions, the constraint \eqref{eq:positiv} is satisfied over the whole domain $\Omega=X$ for a large class of boundary conditions $g_i$ and we then choose $m=n=2$ as shown in \cite[Theorem 4]{AN-ARMA-01}. For instance, we may choose $g_j(x_1,x_2)=x_j$ on $\partial X$ and we are then guaranteed that $(x_1,x_2)\mapsto (u_1,u_2)$ is a  diffeomorphism from $X$ to its image so that the determinant is clearly signed, and by continuity (assuming that $\sigma$ is sufficiently smooth so that gradients are continuous functions, which we also assume for the rest of the paper \cite{gt1}), bounded from below by a positive constant.

In three dimensions, there is no known guarantee that the determinant of three solutions with prescribed boundary conditions remains positive throughout the domain $X$ independent of the conductivity $\sigma$. For boundary conditions of the form $g_j(x)=x_j$ on $\partial X$ for $1\leq j\leq 3$, it is known that the determinant can change signs for some conductivities \cite{BMN-ARMA-04}. 

Yet, \eqref{eq:positiv} is an important constraint in the derivation of our results. By means of specific complex geometric optics solutions, we are able to construct boundary conditions $g_i$ such that \eqref{eq:positiv} is valid on subsets of the domain $X$. More precisely, we shall construct $g_i$ for $1\leq i\leq 4$ so that  $\det (S_1(x),S_2(x),S_3(x))\geq c_0>0$ on $\cup_{k=1}^M \Omega_{2k}$ and $\det(S_1(x),S_2(x),S_4(x))\geq c_0>0$ on $\cup_{k=1}^M \Omega_{2k-1}$ with all domains $\Omega_k$ simply connected, open and such that $X\subset \cup_{k=1}^{2M}\Omega_k$. In other words, we can choose $m=4$ boundary conditions such that $X$ is decomposed into a superposition of overlapping simply connected subsets where the determinant of three out of the four vectors is positive. Generically, we denote by $\Omega$ any of the subdomains $\Omega_k$. We shall see that $\sigma$ can be reconstructed on all of $\Omega$ provided that $\sigma$ and all $S_i$ are known at one point $x_0\in\Omega$. After patching local reconstructions together, this provides a (unique) reconstruction in the whole domain $X$.

On each of the subdomains $\Omega\equiv\Omega_k$, the available information is $H_{ij}(x) = S_i(x)\cdot S_j(x)$ for $1\le i\le j\le n$, abbreviated by a function $H:\Omega\to S_n(\Rm)$ ($S_n(\Rm)$ denotes the $n\times n$ real-valued symmetric matrices). Although our goal is ultimately to recover only the conductivity $\sigma$, our approach requires the reconstruction of the vector fields $S_i,\ 1\le i\le m$ and $F$. Since the data $H(x)$ give us $S(x)$ (the matrix with columns given by the vectors $S_i$) up to an $SO_n(\Rm)$-valued function $R(x)$ (i.e., $R(x)$ is a rotation matrix), the unknowns are now the functions $R$ and $F$, which are of dimension $\frac{n(n-1)}{2}$ and $n$, respectively. 
As will be given in more detail in section \ref{sec:local}, the local reconstruction of $R$ and $F$ proceeds as follows: one first derives divergence and curl equations for the $R_j$'s, column vectors of the function $R:\Omega\to SO_n(\Rm)$, with right-hand-sides involving $F$, the $R_j$'s and the data. Using structural properties of $SO_n(\Rm)$ satisfied at every $x\in\Omega$, one is then able to derive an equation for $F$ of the form 
\begin{align}
    F = \frac{1}{n} \left( \frac{1}{2} \nabla\log\det H + \sum_{1\le i,j\le n} ((V_{ij}+V_{ji})\cdot R_i) R_j \right), \quad n=2,3,
    \label{eq:datatoF}
\end{align}
where the vector fields $V_{ij}$ depend only on the data. Plugging this equation back into the divergence/curl system closes the system for the $R_j$'s. From this closed sytem, one is then able to derive gradient-type equations for the $R_j$'s, where the right-hand sides either depend only on the data (case $n=2$), or also depend polynomially on the $R_j$'s (case $n=3$). In either case, these gradient equations can be integrated along segments of the form $[x_0,x]$ for $x,x_0\in\Omega$, parameterized by the curve
\begin{align}
    \gamma_{x_0,x}:[0,1]\ni t\mapsto\gamma_{x_0,x}(t) = (1-t)x_0 + tx \in\Omega,
    \label{eq:segment}
\end{align}
in order to reconstruct locally the $R_j$'s at $x$ from knowledge of their value at $x_0$. Once the $R_j$'s (and therefore the function $R$) are locally reconstructed around $x_0$, one can reconstruct $\sigma(x)$ around $x_0$ from the knowledge of $\sigma(x_0)$ and integrating \eqref{eq:datatoF} along segments $[x_0,x]$.

As we mentioned earlier,  we can set $\Omega = X$ when $n=2$ by choosing illuminations $\bfg=(g_1,g_2)$ such that \eqref{eq:positiv} is satisfied throughout $X$.  Then the above reconstruction procedure yields unique and stable reconstructions, as described in the following theorem. 
\begin{theorem}[2D global uniqueness and stability]\label{thm:stab2d}
    Let $(H, \wtH)$ be two data sets with coefficients in $W^{1,\infty}(X)$ corresponding to the same illumination $\bfg=(g_1,g_2)$ and with determinants satisfying the condition
    \begin{align}
	\inf_{x\in X}(d, \tilde d) \ge c_0>0, \qquad \mbox{ with } \quad d^2(x)=\det H(x),\quad \tilde d^2(x)=\det\tilde H(x).
	\label{eq:condstab2d}
    \end{align}
    Let  $\sigma$ and $\tilde \sigma$ be the corresponding conductivities and assume that $\sigma(x_0)=\tilde\sigma(x_0)$ for some $x_0\in \overline{X}$. Then, we have the following stability estimate
    \begin{align}
	\|\log\sigma-\log\tilde\sigma\|_{W^{1,\infty}(X)} \le C \|H-\wtH\|_{W^{1,\infty}(X)}.
	\label{eq:stab2d}
    \end{align}
    This result also ensures uniqueness, since \eqref{eq:stab2d} implies $H=\wtH \implies \sigma=\tilde\sigma$.
\end{theorem}
Note that theorem \ref{thm:stab2d} does not require the prescription of the $S_i$'s at the point $x_0$ because the data $H$ together with the maximum principle allow us to access their values at appropriate boundary points, as will be seen in section \ref{sec:gradtheta}, paragraph ``reconstruction procedure''.

In the three dimensional case, as mentioned earlier, we do not know of any illumination $\bfg$ that guarantees the constraint \eqref{eq:positiv} throughout $X$. We then work with $m\ge 3$ solutions of \eqref{eq:conductivity} and assume that there exists an open covering $\O = \{\Omega_k\}_{1\le k\le N}$ ($X\subset\cup_{k=1}^N \Omega_k$), a constant $c_0>0$ and a function $\tau:[1,N]\ni i\mapsto \tau (i) = (\tau(i)_1, \tau(i)_2, \tau(i)_3) \in[1,m]^3$, such that
\begin{align}
    \inf_{x\in\Omega_i} \det(S_{\tau(i)_1}(x), S_{\tau(i)_2}(x), S_{\tau(i)_3}(x)) \ge c_0, \quad 1\le i\le N.
    \label{eq:condji}
\end{align}
The next lemma, whose proof is given in the beginning of section \ref{sec:global}, gives an example of such a setting based on the use of complex geometric optics solutions, if one assumes some regularity on $\sigma$. 
\begin{lemma} \label{lem:setup}
    Let $n=3$ and $\sigma\in H^{\frac{9}{2}+\varepsilon}(X)$ for some $\varepsilon>0$, be bounded from below by a positive constant. Then there exists a non-empty open set $G\subset (H^{\frac{1}{2}}(\partial X))^4$ of quadruples of illuminations such that for any $\bfg = (g_1,g_2,g_3,g_4) \in G$, there exists an open cover of $X$ of the form $\{ \Omega_{2i-1}, \Omega_{2i} \}_{1\le i\le N}$ and a constant $c_0 >0$ such that 
    \begin{align}
	\inf_{x\in\Omega_{2i-1}} \det (S_1, S_2, \tilde \epsilon_i S_4) \ge c_0 \qandq \inf_{x\in\Omega_{2i}} \det (S_1, S_2, \epsilon_i S_3) \ge c_0, \quad 1\le i\le N,
	\label{eq:condji_lemma}
    \end{align}
    for $\epsilon_i$ and $\tilde\epsilon_i$ equal to $\pm1$.
\end{lemma}
Using $m\ge 3$ solutions that satisfy assumption \eqref{eq:condji}, the local reconstruction approach can be applied to reconstruct $\sigma$ on each of the $\Omega_i$'s, and some additional work is then required to patch these reconstructions together and show that the global reconstruction scheme ensures unique and stable reconstructions of $\sigma$ over $X$. Global reconstruction procedures are thus presented in section \ref{sec:global} and summarized in the following:

\begin{theorem}[3D global uniqueness and stability] \label{thm:stab3d}
    Let $X\subset\Rm^3$ be an open convex bounded domain, and let two sets of $m\ge 3$ solutions of \eqref{eq:conductivity} generate measurements $(H, \wtH)$ whose components belong to $W^{1,\infty}(X)$. Assume that one can define a couple $(\O, \tau)$ such that \eqref{eq:condji} is satisfied for both sets of solutions $S$ and $\wtS$. Let also $x_0\in\overline{\Omega}_{i_0}\subset\overline{X}$ and $\sigma(x_0)$, $\tilde \sigma(x_0)$, $\{S_{\tau(i_0)_i}(x_0), \wtS_{\tau(i_0)_i}(x_0)\}_{1\le i\le 3}$ be given. 
    %Then there exists a reconstruction procedure $\sigma = \P_{\O, j}(H)$ such that, denoting $\tilde \sigma = \P_{\O,j}(\wtH)$, the following stability estimate holds
Let $\sigma$ and $\tilde\sigma$ be the conductivities corresponding to the measurements $H$ and $\wtH$, respectively. Then we have the following stability estimate:
    \begin{align}
	\|\log\sigma-\log\tilde\sigma\|_{W^{1,\infty}(X)} \le C\big( \epsilon_0 + \|H-\wtH\|_{W^{1,\infty}(X)}\big),
	\label{eq:globstab3d}
    \end{align}
    where $\epsilon_0$ is the error at the initial point $x_0$
    \begin{align*}
	\epsilon_0 = |\log\sigma_0-\log\tilde\sigma_0| + \sum_{i=1}^3 \|S_{\tau(i_0)_i}(x_0) - \wtS_{\tau(i_0)_i}(x_0) \|.
    \end{align*} 
\end{theorem}
The uniqueness and stability estimate results are constructive. 
As we shall see, the reconstruction depends on the choice of curves along which the sets of ordinary differential equations coming from the gradient equations are solved. In the presence of noise-free data, the reconstruction is unique as we just presented. In the presence of noise, the available data may then no longer be in the range of the measurement operator mapping the unknown coefficient $\sigma$ to the measurements $H(x)$. In this case, the reconstruction may depend on the choice of the curve along which integrations are performed. The compatibility conditions that the data need to verify in order for the reconstruction to be independent of the choice of path (Poincar\'e-type lemma) is straightforward in dimension $n=2$ but seems to be more challenging in dimension $n=3$.

Both theorems provide us with Lipschitz constants of stability in dimensions $n=2$ and $n=3$. These estimates show that reconstructions of the conductivity from knowledge of a sufficient number of power densities is a well-posed problem, unlike what is observed in the Calder\'on problem, which consists of reconstructing the diffusion coefficient from knowledge of the Dirichlet to Neumann map \cite{U-IP-09}. Thus UMEIT and UMOT allow us to combine the high contrast of optical and electrical properties of domains with the high resolution capabilities of ultrasound.

%In theory, when the data are generated by true solutions only, the reconstructions should not depend on the choice of curve, as the gradient fields should all be curl-free by construction. When noise is added however, it will be crucial to project the data onto the range of the measurement operator before inverting for $\sigma$. This analysis will be presented in future work.

The rest of the paper is structured as follows. In section \ref{sec:local} we derive the formulas that are necessary for local (global when $n=2$) reconstructions and stability results. Section \ref{sec:global} explains the global reconstruction scheme in the case $n=3$ and includes the proof of theorem \ref{thm:stab3d}. %Section \ref{sec:conclusion} offers some conclusions. 
%
%
%
%
%%
%%%%%%%%%%%%%%%%%%%%%%%%%%%%%%%
\section{Derivation of local reconstruction formulas} 
\label{sec:local}

In this section, we prove Theorem \ref{thm:stab2d} and Theorem \ref{thm:stab3d} in the setting where $X$ is replaced by an open bounded subset $\Omega\subset X\subset\Rm^n$ such that \eqref{eq:positiv} holds. The proof of Theorem \ref{thm:stab3d} in the case of arbitrary $X$ will be concluded in section \ref{sec:global}.

Knowledge of the matrix $H(x)$ together with the determinant condition \eqref{eq:positiv} allow us to reconstruct the matrix-valued function $S(x):=[S_1(x)|\dots|S_n(x)]$ up to an $SO_n(\Rm)$-valued function $R(x) = [R_1(x)|\dots|R_n(x)]$. This fact can be seen for instance by noticing that the orientation-preserving Gram-Schmidt procedure that creates the orthonormal $R_j$'s from the $S_i$'s with $\det R \det S >0 $ only involves coefficients that depend on the inner products $S_i\cdot S_j = H_{ij}$, which are indeed known. Note that $SH^{-\frac12}(x)$ is also a rotation-valued function on $\Omega$.

%Because the Gram-Schmidt procedure is not the only way to achieve this (one could pick for instance $T$ as the positive squareroot of $H^{-1}$), we define in a more generic manner a matrix-valued function $T(x) = \left\{ t_{ij}(x) \right\}_{1\le i,j\le n}$ that satisfies the properties 
More generally, we  take $T(x) = \left\{ t_{ij}(x) \right\}_{1\le i,j\le n}$ 
 a matrix-valued function on $\Omega$ that satisfies the properties
\begin{align} 
    T^T(x) T(x) = H^{-1}(x) \qandq 0< \det T(x) = (\det H(x))^{-\frac{1}{2}}, \quad x\in\Omega.
    \label{eq:Tmatrix}
\end{align}
Two examples are the symmetric  $T=H^{-\frac12}$ and 
the lower-triangular $T$ in \eqref{eq:TGS} below obtained by the Gram Schmidt procedure.
%(the equality $(\det T)^2 = (\det H)^{-1}$ is obvious from the first property, and we add the second requirement that $\det T\ge 0$). 
Denote $\widetilde T$ the matrix constructed from $\widetilde H$ in the same 
manner as $T$. We impose the existence of a constant  $C_T>0$ such that 
\begin{align} \label{eq:CT}
\|T-\widetilde T\|_{W^{1,\infty}(\Omega)} \le C_T
\|H-\widetilde H\|_{W^{1,\infty}(\Omega)}, 
\end{align}
here  $C_T$ only depends on $H$ and $\widetilde H$ and is completely determined by 
the way we construct  $T$ and $\wtT$. 
 The following simple lemma, whose proof is given in section \ref{app:GS}, justifies the above 
stability result for the Gram-Schmidt procedure:
%%--------------------------------------------------------------------------------------------------
\begin{lemma} \label{lem:GS}
    Let $H, \wtH \in W^{1,\infty}(\Omega)$ that are invertible over $\overline \Omega$ and set
\begin{align*}
m_0= \max\left( \|\wtH^{-1}\|_\infty, \|H^{-1}\|_\infty\right),
M_0= \max\left( \|\wtH^{-1}\|_\infty \|\wtH\|_{W^{1,\infty}(\Omega)}, 
\|H^{-1}\|_\infty \|H\|_{W^{1,\infty}(\Omega)}\right).
\end{align*}
If $(T, \wtT)$ are built from $(H, \wtH)$ using the Gram-Schmidt procedure, then the 
stability estimate holds
\begin{equation}
\|T-\widetilde T\|_{W^{1,\infty}(\Omega)} \le m_0^{3/2}C_n(M_0)
\|H-\widetilde H\|_{W^{1,\infty}(\Omega)}, 
\end{equation}
where $C_n(\xi)$ is a real  polynomial function in $\xi$ that only depends 
on the dimension $n$. 
\end{lemma}
%%--------------------------------------------------------------------------------------------------
 We also define $T^{-1} = \{t^{ij}\}_{1\le i,j\le n}$ and the vector fields 
\begin{align}
    V_{ij} := \nabla(t_{ik})t^{kj}, \quad \text{ i.e. } \quad  V_{ij}^l := \partial_l(t_{ik})t^{kj},\quad 1\le i,j,l \le n.
    \label{eq:Vijs}
\end{align}
Here and below, repeated indices are summed over. Let
$\widetilde V$ be a vector-valued matrix constructed from $\widetilde H$ as $V$ is 
constructed from $H$.  A simple calculation yields
%For stability purposes, we will need the following requirement that there exists a constant $C_V$ such that for all $1\le i,j\le n$
%We impose the existence of an universal constant  $C_V>0$ such that 
\begin{align}
    %\|V_{ij}\|_{L^\infty(\Omega)} \le C_V \|H\|_{W^{1,\infty}(\Omega)} \qandq
     \| V_{ij} - \widetilde V_{ij} \|_{L^\infty(\Omega)} \le C_V
 \|H-\widetilde H\|_{W^{1,\infty}(\Omega)}, \qquad 1\le i,j\le n.
    \label{eq:stabvij}
\end{align}
Here, $C_V =  \left( \|H^{-1}\|_{\infty}^{1/2} 
\|H\|_{\infty} + \|\wtH^{-1}\|_{\infty}^{1/2}  
\|\wtH\|_{\infty} \right)C_T +\|H^{-1}\|_{\infty} $, with $C_T$ defined in \eqref{eq:CT}. \\

%That there exists such $T$ functions such that property \eqref{eq:stabvij} holds is justified in the following lemma, whose proof is given in appendix \ref{app:GS}.

From $S$ and $T$, let us now build the matrix $R$ by defining
\begin{align}
    R(x) := S(x) T(x)^T, \quad \text{i.e. } R_i(x) = t_{ij}(x) S_j(x), \quad 1\le i\le n,\quad x\in\Omega.
    \label{eq:StoR}
\end{align}
From conditions \eqref{eq:Tmatrix} and the fact that $S(x)^T S(x) = H(x)$ for every $x\in\Omega$, the matrix $R$ \eqref{eq:StoR} satisfies $R^T(x)R(x) = \Imm_n$ for all $x\in\Omega$, as well as $\det R(x) = 1$, thus $R(x)\in SO_n(\Rm)$ for all $x\in\Omega$. Moreover, plugging relation \eqref{eq:StoR} into equation \eqref{eq:divSj} and using the vector calculus identity $\div (fV) = \nabla f\cdot V + f\div V$, we obtain
\begin{align*}
    \div R_i = \div (t_{ij} S_j) = (\nabla t_{ij})\cdot S_j + t_{ij} \div S_j = (\nabla t_{ij}) t^{jk}\cdot R_k - t_{ij} F\cdot S_j.
\end{align*}
Thus the $R_i$'s satisfy the following divergence equation:
\begin{align}
    \div R_i = V_{ik}\cdot R_k - F\cdot R_i, \quad 1\le i\le n.
    \label{eq:divRi}
\end{align}
In a similar manner, one can derive the following curl-type equations for the $R_i$'s:
\begin{align}
    n=2\,: \,[\nabla,R_i] &= [V_{ik},R_k] + [F,R_i], \quad i=1,2, \label{eq:curlRi2d} \\
    n=3\,: \,\curl R_i &= V_{ik}\times R_k + F\times R_i, \quad i=1,2,3.
    \label{eq:curlRi3d}
\end{align}
We now show that the redundancies in the systems \eqref{eq:divRi}-\eqref{eq:curlRi2d} when $n=2$ and \eqref{eq:divRi}-\eqref{eq:curlRi3d} when $n=3$ allow us to derive formula \eqref{eq:datatoF} for $F$ in terms of the $R_i$'s and the data. This formula is successively proved for $n=2$ and $n=3$ using vector calculus identities in each dimension, respectively. 

\begin{remark}[Geometric motivation of the next proofs.]
    The generalization to $n\geq 4$ may be found in \cite{MB}, where the natural tool is differential geometry, in particular connections, exterior derivatives and the Hodge transform. Once condition \eqref{eq:positiv} is assumed, an expression of equation \eqref{eq:datatoF} in the basis $(S_1,\dots,S_n)$ may be proved in general dimension by studying the properties of the dual basis of $(S_1,\dots,S_n)$ with respect to the Euclidean metric. Once \eqref{eq:datatoF} is derived, the divergence and curl equations that we have for the $S_i$'s (or equivalently the $R_i$'s) have their right-hand-sides that only depend on the data $H$ and its partial derivatives, and on the $S_i$'s to zero-th order. Then the work is to pass from this resulting system of PDE's to full gradient equations (i.e. total covariant derivatives) for the unknown basis $(S_1,\dots,S_n)$ or $(R_1,\dots,R_n)$, equivalently.    
\end{remark}

\subsection{Elimination of source term and proof of formula \eqref{eq:datatoF}}

\subsubsection{The case $n=2$} \label{sec:F2d}
First notice that because $R\in SO_2(\Rm)$, we have the relations 
\begin{align} \label{eq:defI2}
    JR_i = \varepsilon_{ij} R_j, \qquad (i,j)\in I_2 := \{(1,2), (2,1)\},   
\end{align}
where $\varepsilon_{12} = - \varepsilon_{21} = 1$ and we have defined $J=\left(\begin{matrix} 0&-1\\1&0\end{matrix}\right)$. In particular, this implies that
\begin{align*}
    \div R_i = [\nabla, JR_i] = \varepsilon_{ij} [\nabla,R_j], \qquad (i,j)\in I_2, 
\end{align*}
as well as the relations, for any vector field $A$, 
\begin{align*}
    [A,R_i] = JA\cdot R_i = -A\cdot JR_i = -\varepsilon_{ij} A\cdot R_j, \qquad (i,j)\in I_2.
\end{align*}
Together with the system of equations \eqref{eq:divRi}-\eqref{eq:curlRi2d}, these relations allow us to get the components of $F$ in the basis $(R_1,R_2)$ (here, $(i,j)\in I_2$):
\begin{align*}
    F\cdot R_i &= -\div R_i + V_{ii}\cdot R_i + V_{ij}\cdot R_j  \\
    &= - \varepsilon_{ij} [\nabla,R_j] + V_{ii}\cdot R_i + V_{ij}\cdot R_j \\
    &= - \varepsilon_{ij}([F,R_j] + [V_{ji},R_i] + [V_{jj},R_j]) + V_{ii}\cdot R_i + V_{ij}\cdot R_j \\
    &= -F\cdot R_i + V_{ji}\cdot R_j - V_{jj}\cdot R_i + V_{ii}\cdot R_i + V_{ij}\cdot R_j,
\end{align*}
and hence the formula 
\begin{align}
    2 F\cdot R_i = - (V_{ii} + V_{jj})\cdot R_i + 2V_{ii}\cdot R_i + (V_{ij}+V_{ji})\cdot R_j,
    \label{eq:FdotR12d}
\end{align}
for $(i,j)\in I_2$.
%Likewise, starting from $F\cdot R_2 = -\div R_2 + V_{21}\cdot R_1 + V_{22}\cdot R_2$, one arrives at
%\begin{align}
%    2 F\cdot R_2 = - (V_{11} + V_{22}) + (V_{12}+V_{21})\cdot R_1 + 2V_{22}\cdot R_2.
%    \label{eq:FdotR22d}
%\end{align}
We then obtain equation \eqref{eq:datatoF} by plugging \eqref{eq:FdotR12d} into the relation 
\begin{math}
    F = (F\cdot R_1)R_1 + (F\cdot R_2)R_2,
\end{math}
and using the following identity%, proved in appendix \ref{app:Liouville}
\begin{align}
    - V_{11} - V_{22} = \nabla\log d,
    \label{eq:Liouville2d}
\end{align}
whose proof may be found in section \ref{app:Liouville} below.

%
%%%%%%%%%%%%%%%%%%%%%%%%
\subsubsection{The case $n=3$} \label{sec:F3d}
We recall the vector calculus identity $\div (A\times B) = B\cdot\curl A - A\cdot\curl B$, which holds for any two smooth vector fields $A$ and $B$, together with the relations
\begin{align}\label{eq:defI3}
    R_i(x)\times R_j(x) = R_k(x), \qquad (i,j,k)\in A_3:=\big\{(1,2,3), (2,3,1), (3,1,2)\big\},
    %\quad R_2(x)\times R_3(x) = R_1(x), \qandq R_3(x)\times R_1(x) = R_2(x),
\end{align}
which hold for all $x\in\Omega$ since $R(x)\in SO_3(\Rm)$. Combining these relations, we have
\begin{align*}
    \div R_i = \div (R_j\times R_k) = R_k\cdot\curl R_j - R_j\cdot\curl R_k, \quad (i,j,k)\in A_3.
\end{align*}
Using equations \eqref{eq:divRi} and \eqref{eq:curlRi3d}, we obtain
\begin{align*}
    V_{im}\cdot R_m - F\cdot R_i &= R_k\cdot (V_{jm}\times R_m + F\times R_j ) - R_j \cdot (V_{km}\times R_m + F\times R_k) \\
    &= 2F\cdot R_i - V_{ji}\cdot R_j + V_{jj}\cdot R_i - V_{ki}\cdot R_k + V_{kk}\cdot R_i. 
\end{align*}
Above, we sum over $m$ but not over $(i,j,k)\in A_3$.
Thus we get 
\begin{align*}
    3F\cdot R_i = -(V_{ii} + V_{jj} + V_{kk})\cdot R_i + 2V_{ii}\cdot R_i + (V_{ij}+V_{ji})\cdot R_j + (V_{ik}+V_{ki})\cdot R_k.
\end{align*}
%Writing the other two equalities 
%\begin{align*}
%    \div R_2 = R_1\cdot\curl R_3 - R_3\cdot\curl R_1 \qandq \div R_3 = R_2\cdot\curl R_1 - R_1\cdot\curl R_2,
%\end{align*}
%and doing similar calculations using equations \eqref{eq:divRi} and \eqref{eq:curlRi3d}, we arrive at
%\begin{align*}
%    3F\cdot R_2 = -(V_{11} + V_{22} + V_{33})\cdot R_2 + (V_{21}+V_{12})\cdot R_1 + 2V_{22}\cdot R_2 + (V_{23}+V_{32})\cdot R_3, \\
%    3F\cdot R_3 = -(V_{11} + V_{22} + V_{33})\cdot R_3 + (V_{13}+V_{31})\cdot R_1 + (V_{23}+V_{32})\cdot R_2 + 2V_{33} \cdot R_3.
%\end{align*}
We finally arrive at \eqref{eq:datatoF} by plugging the three components $F\cdot R_i$ into the relation
\begin{math}
    F = (F\cdot R_1) R_1 + (F\cdot R_2) R_2 + (F\cdot R_3)R_3
\end{math}
and using the following identity, %proved in appendix \ref{app:Liouville}
\begin{align}
    - V_{11} - V_{22} - V_{33} = \nabla\log D.
    \label{eq:Liouville3d}
\end{align}
whose proof may be found in section \ref{app:Liouville} below.

\subsubsection{Proof of identities \eqref{eq:Liouville2d} and \eqref{eq:Liouville3d}} \label{app:Liouville}

We prove these identities in general dimension $n\ge 2$. First notice that for any $k=1\dots n$, the function $T(x)$ satisfies locally the (tautological) relation 
\begin{align*}
    \partial_k T(x) = V^k (x) T(x),
\end{align*}
where we have defined $V^k := \{ \partial_k (t_{il})t^{lj} \}_{i,j} = (\partial_k T) T^{-1}$. Therefore by Liouville's formula, we obtain that
\begin{align*}
    \partial_k (\log \det T) = \text{trace } (V^k) = \sum_{i=1}^n V_{ii}^k.
\end{align*}
Finally noting that $\det T = d^{-1}$ when $n=2$ and $\det T = D^{-1}$ when $n=3$, the identities \eqref{eq:Liouville2d} and \eqref{eq:Liouville3d} are proved component by component.

%
%%
%%%%%%%%%%%%%%%%%%%
\subsection{Resolution in $SO_n(\Rm)$ and stability analyses}

Plugging formula \eqref{eq:datatoF} into \eqref{eq:divRi}, \eqref{eq:curlRi2d} and \eqref{eq:curlRi3d} provides a closed system of equations for the function $R$. It remains to show that $R$ is then uniquely and stably determined by such equations. We again separate the cases $n=2$ and $n=3$. %We recall that the set of direct permutations $A_n$ of $(1,\ldots,n)$ is defined in \eqref{eq:defI2} and \eqref{eq:defI3}. %The task is now to parameterize $SO_n(\Rm)$ properly and find gradient-type equations for the parameters that define $R(x)$. 

\subsubsection{The case $n=2$} \label{sec:gradtheta}
Since $SO_2(\Rm)$ is a one-dimensional manifold, $R = [R_1|R_2]$ is described by a $\Sm^1$-valued function $\theta(x)$ such that, at each point, $R_1(\theta) = (\cos\theta,\sin\theta)^T$ and $R_2(\theta) = JR_1(\theta)$. We wish to derive an equation for $\nabla \theta$.
Plugging the expression \eqref{eq:datatoF} of $F$ into \eqref{eq:divRi}, we arrive at
\begin{align}
    \begin{split}
	\div R_i &= \frac{1}{2} \left[ -(\nabla\log d)\cdot R_i + (V_{ij}-V_{ji})\cdot R_j \right], \quad (i,j)\in I_2. %\\ \div R_2 &= \frac{1}{2} \left[ - (V_{12}-V_{21})\cdot R_1 -(\nabla\log d)\cdot R_2 \right].
    \end{split}
    \label{eq:divRiF2d}
\end{align}

Let us now derive a differential equation for $R = [R_1|R_2]$. The relation $R^T R=\Imm_2$ implies that $R^T \partial_i R \in \A_2(\Rm)$, the space of two-dimensional anti-symmetric matrices. For $i=1,2$, i.e., it can be written in the form $R^T \partial_i R = \alpha_i J$, where $\alpha_i = R_2\cdot \partial_i R_1$. Defining $\bm\alpha:= \alpha_i \bfe_i$, it is clear that $\bm\alpha\cdot Z = R_2\cdot[(Z\cdot\nabla)R_1]$ for any vector field $Z$. For the sequel, we need the following vector calculus identity, which holds for any smooth vector field $A$
\begin{align}
    \nabla |A|^2 &= 2(A\cdot\nabla)A - 2[\nabla,A]JA.
    \label{eq:vci2d2}
\end{align}
We decompose $\bm\alpha$ in the basis $(R_1, R_2)$ and use equations \eqref{eq:divRi} 
\begin{align*}
    \bm\alpha &= (\bm\alpha\cdot R_1)R_1 + (\bm\alpha\cdot R_2)R_2 
    = (R_2\cdot [(R_1\cdot\nabla) R_1]) R_1 + (R_2\cdot [(R_2\cdot\nabla) R_1]) R_2 \\
    &= (R_2\cdot [(R_1\cdot\nabla) R_1]) R_1 - (R_1\cdot [(R_2\cdot\nabla) R_2]) R_2.
\end{align*}
We now use \eqref{eq:vci2d2}, which, in this case, becomes $(R_i\cdot\nabla)R_i = [\nabla, R_i] JR_i = - (\div R_j) R_j$ for $(i,j)\in I_2$, %and similarly $(R_2\cdot\nabla)R_2 = -[\nabla,R_2]R_1$, 
and then equation \eqref{eq:divRiF2d} to obtain that
\begin{align*}
    \bm\alpha &= - (\div R_2) R_1 + (\div R_1) R_2 \\
    &= - \frac{1}{2} (-(V_{12}-V_{21})\cdot R_1 - (\nabla\log d)\cdot R_2) R_1 + \frac{1}{2}( -(\nabla\log d)\cdot R_1 + (V_{12}-V_{21})\cdot R_2) R_2 \\
    &= \frac{1}{2} [V_{12}-V_{21} - (R_2\otimes R_1 - R_1\otimes R_2)\nabla\log d] = \frac{1}{2} [V_{12}-V_{21} - J\nabla\log d].
\end{align*}
Finally, given the above parameterization $R(\theta)$, one checks using the chain rule that $R^T \partial_i R = (\partial_i\theta) J,\ i=1,2$. Using the preceding calculations, we obtain that 
\begin{align}
    \nabla\theta = \bm\alpha = \frac{1}{2} [V_{12}-V_{21} - J\nabla\log d].
    \label{eq:gradtheta}
\end{align}

\paragraph{Reconstruction procedure.}

Let $H$ be a given data set corresponding to an illumination $\bfg$ that guarantees condition \eqref{eq:positiv} throughout $X$. From $S$, we construct $R$ using the Gram-Schmidt procedure, and parameterize $R$ with a function $\theta:X\to\Sm^1$ as above. Let $x_m\in\partial X$ be the global minimum of the illumination $g_1$. Then according to the maximum principle \cite{gt1}, $u_1$ achieves its minimum over $X$ at $x_m$. In particular, at this point, we have the relation
\begin{align*}
    \frac{\nabla u_1}{|\nabla u_1|}(x_m) = -\nu(x_m) = \frac{S_1(x_m)}{|S_1(x_m)|} = R_1(\theta(x_m)),
\end{align*}
where $\nu(x_m)$ denotes the outgoing normal vector to $X$ at $x_m$. Thus $\theta(x_m)$ is known. Therefore, we can reconstruct $\theta(x)$ at every $x\in X$ by integrating \eqref{eq:gradtheta} along the segment $[x_m,x]$ parameterized by $\gamma_{x_m,x}$:
\begin{align*}
    \theta(x) &= \theta(x_m) + \int_0^1 \dot\gamma_{x_m,x}(t)\cdot \nabla\theta (\gamma_{x_m,x}(t))\ dt  \\
    &= \theta(x_m) + \frac{1}{2}(x-x_m)\cdot\int_0^1 (V_{12}-V_{21} - J\nabla\log d) (\gamma_{x_m,x}(t))\ dt.
\end{align*}
Once $\theta$ is recovered throughout $X$, one then reconstructs $\sigma(x)$ for all $x\in X$ from the knowledge of $\sigma(x_0)$ for some $x_0\in X$ and integrating equation \eqref{eq:datatoF} along the segment $[x_0,x]$, that is
\begin{align*}
    \log\sigma(x) &= \log\sigma(x_0) + 2\int_0^1 \dot\gamma_{x_0,x}(t)\cdot F(\gamma_{x_0,x}(t))\ dt \\
    &= \log\sigma(x_0) + (x-x_0)\cdot\int_0^1 (\nabla\log d + ((V_{ij} + V_{ji})\cdot R_i)R_j) (\gamma_{x_0,x}(t))\ dt.
\end{align*}

\paragraph{Proof of Theorem \ref{thm:stab2d}}
We now prove the stability results stated in Theorem \ref{thm:stab2d}.
%\begin{proof}[Proof of proposition \ref{thm:stab2d}]
    Let two data sets $(H,\wtH)\in W^{1,\infty}(X)$ correspond to identical illuminations $\bfg = \tilde\bfg$, and conductivities $\sigma$ and $\tilde\sigma$ that coincide at some $x_0\in \overline X$. Let $T, \wtT$ be built from $H, \wtH$ such that they satisfy conditions \eqref{eq:Tmatrix} and \eqref{eq:stabvij}. We then define $R = S T^T$ and $\wtR = \wtS \wtT^T$ and parameterize $R$ and $\wtR$ by their angle functions $\theta,\tilde\theta:X\to \Sm^1$ as above. 
    Using estimates \eqref{eq:stabvij} and taking the difference of \eqref{eq:gradtheta} for $\theta$ and $\tilde\theta$ yields
    \begin{align}
	2|\nabla(\theta-\tilde\theta)| \le |V_{21}-\wtV_{21}| + |V_{12}-\wtV_{12}| + |\nabla(\log d-\log\tilde d)| \le C \|H-\wtH \|_{W^{1,\infty}(X)}.
	\label{eq:nablathetadiff}
    \end{align}
    Moreover, if $x_m\in\partial X$ is the global minimum of the illumination $g_1$, we have seen in the reconstruction procedure that $R(\theta(x_m)) = \wtR(\tilde\theta(x_m)) = -\nu(x_m)$ and thus $\theta(x_m) = \tilde \theta(x_m)$. Therefore we have, for any $x\in\overline X$, 
    \begin{align*}
	\theta(x)-\tilde\theta(x) &= \theta(x_m)-\tilde\theta(x_m) + (x-x_m)\cdot\int_0^1 \nabla(\theta-\tilde\theta) ( \gamma_{x_m,x}(t))\ dt,
    \end{align*}
    and thus $|\theta(x)-\tilde\theta(x)| \le \Delta(X) \|\nabla(\theta-\tilde\theta)\|_{L^\infty}(X)$, where $\Delta(X)$ denotes the diameter of $X$. Combined with \eqref{eq:nablathetadiff}, this yields
    \begin{align}
	\|\theta-\tilde\theta\|_{W^{1,\infty}(X)} \le C_\theta \|H- \tilde H\|_{W^{1,\infty}(X)}.
	\label{eq:stabtheta}
    \end{align}
    This inequality is also a stability statement for the reconstruction of the function $\theta$. On to the stability of $\sigma$, we have the pointwise estimate
    \begin{align*}
	|\nabla (\log\sigma-\log\tilde\sigma)| &\le |\nabla(\log d-\log\tilde d)| + |((V_{pk}-\wtV_{pk} + V_{kp}-\wtV_{kp})\cdot R_k) R_p| \\
	&\quad + |(\wtV_{pk}+\wtV_{kp})\cdot (R_k-\wtR_k) R_p| + |((\wtV_{pk}+\wtV_{kp})\cdot \wtR_k) (R_p-\wtR_p)| \\
	&\le |\nabla(\log d-\log\tilde d)| + 2 \sum_{1\le k,p\le 2} |V_{pk}-\wtV_{pk}| + 4 |\theta-\tilde\theta| \sum_{1\le k,p\le 2} |\wtV_{pk}|,
    \end{align*}
    where we have used the Cauchy-Schwarz inequality and the fact that $|\partial_\theta R_p|\le 1$ for $p=1,2$. Using estimates \eqref{eq:stabtheta} and \eqref{eq:stabvij} implies directly 
    \begin{align*}
	\|\nabla(\log\sigma-\log\tilde\sigma)\|_{L^\infty(X)} \le C_1 \|H- \wtH\|_{W^{1,\infty}(X)}.
    \end{align*}
    The proposition is proved provided that $\sigma$ and $\tilde\sigma$ coincide at $x_0$, which allows us to control any difference $|\log\sigma(x)-\log\tilde\sigma(x)|$ by $\Delta(X)\|\nabla(\log\sigma-\log\tilde\sigma)\|_{L^\infty(X)}$.   This concludes the proof.
%\end{proof}

\subsubsection{The case $n=3$} \label{sec:SO3}
In the three-dimensional case, the unknown is now an $SO_3(\Rm)$-valued function $R$ and is thus parameterized by  $3$ scalar functions in principle. However, because every $3$-parameter chart on $SO_3(\Rm)$ has singularities, as shown in \cite{S-SIAMREV-64}, using such a chart to represent the unknown function $R$ may yield instabilities during the reconstruction procedure. In this paper, we thus represent $R$ as a $9$-vector $\bfR = (R_1^T, R_2^T, R_3^T)^T$. We then derive gradient equations for the components of $\bfR$, which allow local reconstruction of $\bfR$ by integration of an appropriate system of ordinary differential equations.

\begin{remark} \label{rem:quaternion}
    For numerical purposes, we want to choose a parameterization of $SO_3(\Rm)$ that involves as few parameters as possible while still guaranteeing stability. This can be achieved by the $4$-parameter quaternion representation, whose details will appear elsewhere. %, and shall be used for subsequent numerical simulations. 
\end{remark}

Let us now find differential equations for $R$. Taking partial derivatives of the relation $R^T R = \Imm_3$ yields that for $k=1,2,3$, we have that $R^T \partial_k R\in \A_3(\Rm)$, the space of antisymmetric $3\times 3$ matrices, which we write 
\begin{align}
    R^T \partial_k R = \left[
	\begin{array}{ccc}
	    0 & \alpha_3^k(x) & -\alpha_2^k(x) \\ -\alpha_3^k(x) & 0 & \alpha_1^k(x) \\ \alpha_2^k(x) & -\alpha_1^k(x) & 0
	\end{array}
	\right], \quad k=1,2,3.
	\label{eq:asym}
\end{align}
We now focus on expressing the vector fields $\bm\alpha_i := \alpha_i^k \bfe_k$ in terms of the data. In order to do so, we recall that $A_3$ is the set of direct permutations of $(1,2,3)$. We notice that 
\begin{align}
    \bm\alpha_i\cdot R_p = R_j \cdot [(R_p\cdot\nabla) R_k], \quad (i,j,k) \in A_3, \quad p\in \{i,j,k\}. 
    \label{eq:identity3d}
\end{align}
We now use the following identities, which hold for any smooth vector fields $A$ and $B$
{\small\begin{align*}
    2(B\cdot\nabla)A &= \nabla(A\cdot B) + \curl(A\times B) - A\times(\curl B) - B\times(\curl A) - (\div B)A + (\div A)B, \\
    2 (A\cdot\nabla) A &= \nabla |A|^2 - 2 A\times (\curl A),
\end{align*}}
where $A = R_k$, $B = R_p$, and we take the dot product with $R_j$. For $p=i$, we obtain
\begin{align}
    2 \bm\alpha_i\cdot R_i &= 2 R_j\cdot[(R_i\cdot\nabla)R_k] \nonumber\\
    &= R_j\cdot [\curl R_j - R_k \times (\curl R_i) - R_i \times(\curl R_k) - (\div R_i) R_k + (\div R_k) R_i] \nonumber\\
    &= R_j\cdot (\curl R_j) - R_i\cdot (\curl R_i) + R_k\cdot (\curl R_k). \label{eq:alphaiRi}
\end{align}
For $p=j$, we get
\begin{align}
    2 \bm\alpha_i\cdot R_j &= 2 R_j\cdot[(R_j\cdot\nabla)R_k] = -2 R_k\cdot [(R_j\cdot\nabla)R_j] = 2 R_k\cdot [R_j\times(\curl R_j)] \nonumber \\
    &= -2 R_i\cdot (\curl R_j). \label{eq:alphaiRj}
\end{align}
Finally for $p=k$, we get
\begin{align}
    2\bm\alpha_i\cdot R_k = 2 R_j\cdot [(R_k\cdot\nabla) R_k] = -2 R_j \cdot [R_k\times(\curl R_k)] = -2 R_i \cdot (\curl R_k). 
    \label{eq:alphaiRk}
\end{align}
We have just proved that the vector fields $\bm\alpha_l$ can be expressed in terms of the quantities $\{ R_p\cdot(\curl R_q)\}_{1\le p,q\le 3}$. These quantities, in turn, can be expressed in terms of $R$ and the data using equation \eqref{eq:curlRi3d}. For $(i,j,k)\in A_3$, we have
\begin{align}
    \begin{split}
	(\curl R_i)\cdot R_i &= (V_{il}\times R_l + F\times R_i)\cdot R_i = - V_{ij}\cdot R_k + V_{ik}\cdot R_j \\
	(\curl R_i)\cdot R_j &= (V_{il}\times R_l + F\times R_i)\cdot R_j = V_{ii}\cdot R_k - V_{ik}\cdot R_i + F\cdot R_k \\  
	(\curl R_i)\cdot R_k &= (V_{il}\times R_l + F\times R_i)\cdot R_k = -V_{ii}\cdot R_j + V_{ij}\cdot R_i - F\cdot R_j.	
    \end{split}
    \label{eq:curlRiRj}    
\end{align}

To summarize, using \eqref{eq:datatoF} for $F$ and  identity \eqref{eq:Liouville3d}, we have for $(i,j,k)\in A_3$ that
\begin{align}
    \begin{split}
	\bm\alpha_i\cdot R_i &= \frac{1}{2} \left[ (V_{kj}-V_{jk})\cdot R_i - (V_{ik} + V_{ki})\cdot R_j + (V_{ji}+V_{ij})\cdot R_k \right], \\
	\bm\alpha_i\cdot R_j &= \frac{1}{3} \left[ (V_{ik} + V_{ki})\cdot R_i + (V_{kj} - 2 V_{jk})\cdot R_j + (2 V_{jj} + V_{kk} - V_{ii})\cdot R_k \right], \\
	\bm\alpha_i\cdot R_k &= \frac{1}{3} \left[ -(V_{ij}+V_{ji})\cdot R_i - (V_{jj} + 2V_{kk} - V_{ii})\cdot R_j + (2V_{kj}-V_{jk})\cdot R_k \right].
    \end{split}
    \label{eq:alphas}
\end{align}
These calculations show that there exist vector fields $A_{pqr}$ such that 
\begin{align}
    \bm\alpha_p = (A_{pqr}\cdot R_q)R_r, \qquad\text{i.e.}\qquad \alpha_p^k = A_{pqr}^m R_q^m R_r^k,\qquad  k=1,2,3.
    \label{eq:alphas2}
\end{align}

Left-multiplying equation \eqref{eq:asym} by $R$, we obtain the differential equations
\begin{align}
    \partial_k \bfR = \left[
    \begin{array}{c}
	\partial_k R_1 \\
	\partial_k R_2 \\
	\partial_k R_3
    \end{array}
    \right] = \left[
    \begin{array}{c}
	- \alpha_3^k R_2 + \alpha_2^k R_3 \\
	-\alpha_1^k R_3 + \alpha_3^k R_1 \\
	-\alpha_2^k R_1 + \alpha_1^k R_2	
    \end{array}
    \right], \quad 1\le k\le 3.
    \label{eq:gradR}
\end{align}
Since the expression \eqref{eq:alphas2} is purely quadratic in the components of $\bfR$, the above system can be summarized as
\begin{align}
    \partial_k \bfR = \sum_{|\alpha|=3} Q_\alpha^k \bfR^\alpha, \quad \left( \bfR^\alpha = \prod_{i=1}^9 R_i^{\alpha_i} \right),
    \label{eq:gradR1}
\end{align}
where the functions $Q_\alpha^k:\Omega\to \Rm^9$ are linear combinations of the data vectors $V_{ij}$ \eqref{eq:Vijs}. 

%
%%%%%%%%%%%%%%%%%%%%%%%%%%
\paragraph{Local reconstructibility and stability estimates.}

The system \eqref{eq:gradR1} allows us to reconstruct $\bfR$ locally provided that we know its value $\bfR(x_0) = \bfR_0$ for some $x_0 \in \Omega$. The reconstruction of $\bfR(x)$ for $x\in\Omega$ is done by integrating the following ODE
\begin{align}
    \frac{d}{dt} \bfR(\gamma(t)) = G(\gamma(t), \bfR(\gamma(t))), \quad \bfR(\gamma(0)) = \bfR_0,
    \label{eq:ODER}
\end{align}
where $\gamma$ is chosen to be $\gamma_{x_0,x}$ here, and where the function $G$ is polynomial in the components of $\bfR$ with bounded coefficients provided that $H$ has components in $W^{1,\infty}$. Because $G$ is constructed such that $G\cdot\bfR=0$, as can be seen in \eqref{eq:gradR}, the solution to this system of equations, if it exists, has constant $\|\bfR\|^2$ norm, equal to $\|\bfR_0\|^2=3$. Therefore, the function $\bfR(x) \in \Rm^9$ remains in the compact set $\sqrt{3}\Sm^8$, over which $G$ is Lipschitz in the $\bfR$ variable. Thus by virtue of standard results for ordinary differential equations \cite{HS-74}, the solution to \eqref{eq:ODER} exists, is unique, and its existence can be extended up to $t=1$. Such a reconstruction procedure is stable in the sense of the following proposition.

\begin{proposition} [Local stability when $n=3$]\label{prop:stab3d}
    Let $\Omega\subset X$ and $H, \wtH:\Omega\to S_3(\Rm)$ be two data sets with components in $W^{1,\infty}(\Omega)$ that satisfy the condition
    \begin{align}
	\inf_{x\in\Omega}(D, \wtD) \ge c_0>0.
	\label{eq:condstab3d}
    \end{align}
    Then there exist two constants $C_0,C_1$ such that for every $x,y \in\Omega$, we have the estimate 
    \begin{align}
	\|\bfR - \widetilde\bfR\|(y) \le C_0 \|\bfR - \widetilde\bfR\|(x) + C_1 \|H-\wtH\|_{W^{1,\infty}(\Omega)}. 
	\label{eq:localstab}
    \end{align}
\end{proposition}

\begin{proof}[Proof of proposition \ref{prop:stab3d}]
    Let two data sets $H,\wtH:\Omega\to S_3(\Rm)$ with components in $W^{1,\infty}(\Omega)$ correspond to conductivities $(\sigma,\tilde\sigma)$. Let $T, \wtT$ be built from $H, \wtH$ such that they satisfy conditions \eqref{eq:Tmatrix} and \eqref{eq:stabvij}. We then define $R = ST^T$ and $\wtR = \wtS \wtT^T$.  For the rest of the proof, we denote by $\bfR = (R_1^T, R_2^T, R_3^T)^T$ and $\widetilde\bfR$ similarly. Recall that we have equations of the type
    \begin{align*}
	\partial_k \bfR = \sum_{|\alpha|=3} Q_{\alpha}^k \bfR^\alpha,\quad 1\le k\le 3,
    \end{align*}
    where the $Q_{\alpha}^k$'s are linear combinations of $\{V_{ij}^l\}$ \eqref{eq:Vijs}. Given the stability condition \eqref{eq:stabvij}, we obtain estimates of the form 
    \begin{align}
	\max_{\alpha, k} \|\wtQ_\alpha^k\|_{L^\infty(\Omega)} \le C \|\wtH\|_{W^{1,\infty}(\Omega)} \qandq \max_{\alpha,k} \|Q_{\alpha}^k - \wtQ_\alpha^k\|_{L^\infty(\Omega)} \le C \|H-\wtH\|_{W^{1,\infty}(\Omega)}.
	\label{eq:stabQ}
    \end{align}
    We then write the pointwise relation
    \begin{align*}
	\partial_k \|\bfR-\widetilde\bfR\|^2 &= 2(\bfR-\widetilde\bfR)\cdot\partial_k (\bfR-\widetilde\bfR) = 2 (\bfR-\widetilde\bfR)\cdot \sum_{|\alpha|=3} [Q_\alpha^k \bfR^\alpha - \wtQ_\alpha^k \widetilde\bfR^\alpha] \\
	&= 2 (\bfR-\widetilde\bfR)\cdot \sum_{|\alpha|=3} [(Q_\alpha^k - \wtQ_\alpha^k) \bfR^\alpha + \wtQ_\alpha^k (\bfR^\alpha-\widetilde\bfR^\alpha)] .
    \end{align*}
    Recall that $\bfR$ and $\widetilde\bfR$ satisfy $\|\bfR\|^2= \|\widetilde\bfR\|^2=3$ on $\Omega$ so each of their components is bounded by $\sqrt{3}$. Therefore, for any $9$-index $\alpha$ with $|\alpha|=3$,
  {\small  \begin{align*}
	|\bfR^\alpha-\widetilde\bfR^\alpha| = \left|\prod_{l=1}^3 R_{i_l} - \prod_{l=1}^3 \wtR_{i_l}\right| &\le |(R_{i_1}-\wtR_{i_1})R_{i_2}R_{i_3}| + |\wtR_{i_1}(R_{i_2}-\wtR_{i_2})R_{i_3}| + |\wtR_{i_1}\wtR_{i_2}(R_{i_3}-\wtR_{i_3})| \\
	&\le 9\|\bfR-\widetilde\bfR\|.	
    \end{align*}}
    Therefore we have the estimate
    \begin{align*}
	\partial_k \|\bfR-\widetilde\bfR \|^2 \le 2 \|\bfR-\widetilde\bfR\| \sum_{|\alpha|=3} [\|Q_{\alpha}^k-\wtQ_{\alpha}^k\|_\infty + 9\| \wtQ_{\alpha}^k \|_\infty \|\bfR-\widetilde\bfR\|].
    \end{align*}
    Writing $\partial_k \|\bfR-\widetilde\bfR\|^2 = 2 \|\bfR-\widetilde\bfR\| \partial_k \|\bfR-\widetilde\bfR\|$ and using the estimates in \eqref{eq:stabQ}, we obtain for some constants $C_0$ and $C_1$ that
    \begin{align*}
	\partial_k \|\bfR-\widetilde\bfR\| \le C_0 \|H-\wtH\|_\infty + C_1 \|\wtH\|_\infty \|\bfR-\widetilde\bfR\|.
    \end{align*}
    Let $x,y\in\Omega$. Then applying Gronwall's lemma for the function $\|\bfR-\widetilde\bfR\|$ over the segment $[x,y]$ parameterized by $\gamma_{x,y}$, we get the estimate
    \begin{align}
	\|\bfR-\widetilde\bfR\|(y) \le ( \|\bfR-\widetilde\bfR\|(x) + \Delta(\Omega) C_0 \|H-\wtH\|_{1,\infty}) e^{C_1 \Delta(\Omega) \|\wtH\|_{1,\infty}}.
	\label{eq:f_estim}
    \end{align}
    This concludes the proof.
\end{proof}

\subsubsection{Gram-Schmidt decomposition}
\label{app:GS}
The Gram-Schmidt decomposition is defined as follows
\begin{align*}
    R_1 := S_1/|S_1| \qandq R_j = \left( S_j - \sum_{i=1}^{j-1} (S_j\cdot R_i)R_i \right) / \left|S_j - \sum_{i=1}^{j-1} (S_j\cdot R_i)R_i\right|, \quad j > 1,
\end{align*}
and builds a matrix $R$ from $S$ that is orthogonal by construction.  In the three-dimensional case, the transition matrix $T$ such that $R=S T^T$ is given by:
\begin{align}
    T = \{t_{ij}\}_{1\le i,j\le 3} &= \left[
    \begin{array}{ccc}
	H_{11}^{-\frac{1}{2}} & 0 & 0 \\
	-H_{12} H_{11}^{-\frac{1}{2}} d^{-1} & H_{11}^{\frac{1}{2}} d^{-1} & 0 \\
	(H_{12}H_{23}- H_{22}H_{13}) (dD)^{-1} & (H_{12}H_{13}-H_{11}H_{23}) (dD)^{-1} & d D^{-1}
    \end{array}
    \right], \label{eq:TGS} \\
    &\text{with}\quad d := (H_{11} H_{22} - H_{12}^2)^\frac{1}{2} \qandq D = (\det H)^{\frac{1}{2}}. \nonumber 
\end{align}
The vector fields $V_{ij}$ defined in \eqref{eq:Vijs} take the following expression
\begin{align}
    \{V_{ij} \}_{1\le i,j\le 3} = \left[
    \begin{array}{ccc}
	\nabla\log t_{11} & 0 & 0 \\
	\frac{t_{22}}{t_{11}}\nabla\frac{t_{21}}{t_{22}} & \nabla\log t_{22} & 0 \\
	\frac{t_{33}}{t_{11}} \nabla\frac{t_{31}}{t_{33}} - \frac{t_{21}t_{33}}{t_{11}t_{22}}\nabla\frac{t_{32}}{t_{33}} & \frac{t_{33}}{t_{22}}\nabla\frac{t_{32}}{t_{33}} & \nabla\log t_{33}
    \end{array}
    \right].
    \label{eq:VijGS}
\end{align}
When $n=2$, $T$ and $V_{ij}$ are given by the top-left $2\times 2$ blocs of \eqref{eq:TGS} and \eqref{eq:VijGS}, respectively.

\begin{proof}[Proof of lemma \ref{lem:GS}]
For simplicity we only prove the Lemma  in the case where the dimension is  two. The three
dimensions can be done similarly.   We first notice that $\|H^{-1}(x)\|_2  \le  \|H^{-1}\|_\infty$
over $\Omega$. Since $H$ is a symmetric matrix we have 
\begin{align*}
  \|H^{-1}(x)\|_2^{-1}\leq  H_{jj}(x), 
\end{align*}
and so $  \|H^{-1}\|_\infty^{-1} \leq  H_{jj}(x)$   over $\Omega$ for $1\le j\le 2$ . Using these estimates and
Remark~\ref{rem:Hmun} we easily derive the inequalities
\begin{align*}
\|t_{ij} -\tilde t_{ij} \|_\infty \leq 2 m_0^{3/2}M_0 
\|H-\widetilde H\|_{W^{1,\infty}(\Omega)},\\
\|\nabla t_{ij} -\nabla \tilde t_{ij} \|_\infty \leq4 m_0^{3/2}
\sum_{j=0}^4 M_0^j  
\|H-\widetilde H\|_{W^{1,\infty}(\Omega)},
\end{align*}   
for $1\le i,\, j\le 2$,  which proves the result. 
\end{proof}

%
%%
%%%%%%%%%%%%%%%%%%%%%%%
\section{Global reconstructions in 3D} \label{sec:global}

As we have seen in the past section, the three-dimensional approach is locally similar to the two-dimensional case, as far as the pointwise derivation of reconstruction formulas is concerned. The main difference is that, because the result \cite[Theorem 4]{AN-ARMA-01} does not hold when $n=3$, one must work with a covering of $X$ and build a global reconstruction procedure that patches the local reconstructions over the domains $\Omega_i$ together in a stable manner. %This is what we will do now. 

%
%%%%%%%%%%%%%%%%%%%%%
\subsection{Proof of lemma \ref{lem:setup}}
Let us first justify that the approach presented in the next section is valid in the sense that we can build illuminations such that a couple $(\O, j)$ satisfies condition \eqref{eq:condji}. 

\begin{proof}[Proof of lemma \ref{lem:setup}] A way to fulfill condition \eqref{eq:condji} is to construct solutions whose gradients are taylored to satisfy this condition up to terms that can be made negligible. We do this by using the Complex Geometrical Optic (CGO) solutions, a generalization of the harmonic complex plane waves of the form $e^{\bm\rho\cdot x}$ with $\bm\rho\in\Cm^n$ such that $\bm\rho\cdot\bm\rho=0$, so that $\Delta e^{\bm\rho\cdot x} = \bm\rho\cdot\bm\rho e^{\bm\rho\cdot x}=0$. These functions were first introduced in \cite{calderon80} in the context of linearized inverse problems, then extended in \cite{Syl-Uhl-87} in the context of non-linear inverse problems. The construction can be made for any $n\ge 2$. \\
    {\bf Construction of CGO's: } We first extend the diffusion equation $\nabla\cdot(\sigma(x)\nabla u) = 0$ to $\Rm^n$, where $\sigma(x)$ is extended in a continuous manner outside of $X$ and such that $\sigma\equiv 1$ outside of a large ball. Assuming that $\sigma|_{X}\in H^{\frac{n}{2}+3+\varepsilon}(X)$ for some $\varepsilon>0$ (this implies $\sigma\in\C^3(\overline X)$ by Sobolev imbedding), it is shown in \cite[Corollary 3.2]{BU-IP-10} following works in \cite{calderon80,Syl-Uhl-87}, that there exist complex-valued solutions of the above full-space diffusion equation of the form
    \begin{align}
	u_{\bm\rho} = \frac{1}{\sqrt{\sigma}} e^{\bm\rho\cdot x} (1 + \psi_{\bm\rho}), 
	\label{eq:ubmrho}
    \end{align}
    where $\bm\rho\in\Cm^n$ is a complex frequency satisfying $\bm\rho\cdot\bm\rho = 0$, which is equivalent to taking $\bm\rho = \rho (\bfk + i\bfk^\perp)$ for some $\bfk, \bfk^\perp \in \Sm^{n-1}$ such that $\bfk\cdot\bfk^\perp=0$ and $\rho = |\bm\rho|/\sqrt{2} >0$. Moreover, the remainder $\psi_{\bm\rho}$ satisfies the PDE
    \begin{align}
	\Delta \psi_{\bm\rho} + 2 \bm\rho\cdot\nabla\psi_{\bm\rho} = \frac{\Delta \sqrt{\sigma}}{\sqrt{\sigma}} (1+\psi_{\bm\rho}).
	\label{eq:psirho}
    \end{align}
    Energy estimates on \eqref{eq:psirho} done in \cite{BU-IP-10} use the regularity assumed on $\sigma$ to show that $\rho \psi_{\bm\rho}|_X = O(1)$ in $\C^1(\overline X)$. Using this estimate, computing the gradient of \eqref{eq:ubmrho} and rearranging term, we arrive at
    \begin{align*}
	\sqrt{\sigma}\nabla u_{\bm\rho} = e^{\bm\rho\cdot x} (\bm\rho + \bm\varphi_{\bm\rho}), \quad\text{with}\quad \bm\varphi_{\bm\rho} := \nabla\psi_{\bm\rho} + \psi_{\bm\rho} \bm\rho - (1+\psi_{\bm\rho}) \nabla\sqrt{\sigma}. 
    \end{align*}
    Because $\nabla \sqrt \sigma$ is bounded and $\rho{\psi_{\bm\rho}}|_{X} = O(1)$ in $\C^1(\overline X)$, the $\Cm^n$-valued function $\bm\varphi_{\bm\rho}$ satisfies $\sup_{\overline X} | \bm\varphi_{\bm\rho}| \leq C$ independent of $\brho$. Moreover, the constant $C$ is in fact independent of $\sigma$ provided that the norm of the latter is bounded by a uniform constant in  $H^{\frac n2+3+\eps}(X)$. \\
    {\bf Completion of the proof:} Note that $u_{\bm\rho}$ is complex-valued, thus its real and imaginary parts $u_{\bm\rho}^\Re$ and $u_{\bm\rho}^\Im$ provide two solutions of the diffusion equation, and $\sqrt\sigma \nabla u_{\bm\rho}^\Re$ and $\sqrt\sigma \nabla u_{\bm\rho}^\Im$ may serve as vectors $S_i$. More precisely, we have 
    \begin{align}
	\begin{split}
	    \sqrt\sigma\nabla u_{\bm\rho}^\Re  &= \rho e^{\rho\bfk\cdot x} \left( (\bfk+ \rho^{-1}\bm\varphi_{\bm\rho}^\Re )\cos(\rho\bfk^\perp\cdot x) - (\bfk^\perp+ \rho^{-1}\bm\varphi_{\bm\rho}^\Im ) \sin (\rho\bfk^\perp\cdot x) \right),  \\
	    \sqrt\sigma\nabla u_{\bm\rho}^\Im  &= \rho e^{\rho\bfk\cdot x} \left( (\bfk^\perp + \rho^{-1}\bm\varphi_{\bm\rho}^\Im) \cos(\rho\bfk^\perp\cdot x) + (\bfk + \rho^{-1} \bm\varphi_{\bm\rho}^\Re) \sin (\rho\bfk^\perp\cdot x) \right). 
	\end{split}
	\label{eq:urho}    
    \end{align}
    Let us now define $\bm\rho_1 = \rho (\bfe_2 + i\bfe_1)$, $\bm\rho_2 = \rho (\bfe_3 + i\bfe_1)$, and construct
    \begin{align*}
	(S_1, S_2, S_3, S_4) = \sqrt{\sigma} (\nabla u_{\bm\rho_1}^\Re, \nabla u_{\bm\rho_1}^\Im, \nabla u_{\bm\rho_2}^\Re, \nabla u_{\bm\rho_2}^\Im).
    \end{align*}
    Then, using \eqref{eq:urho}, we obtain that 
    \begin{align*}
	\det(S_1,S_2,S_3) &= \rho^3 e^{\rho(2x_2 + x_3)} \left( - \cos(\rho x_1) + f_1(x) \right), \\
	\det(S_1, S_2, S_4) &= \rho^3 e^{\rho(2x_2 + x_3)} \left( -\sin(\rho x_1) + f_2(x) \right),
    \end{align*}
    where $\lim_{\rho\to\infty} \sup_{\overline X} |f_1| = \lim_{\rho\to\infty} \sup_{\overline X} |f_2| = 0$. Letting $\rho$ so large that $\sup_{\overline X} (|f_1|,|f_2|)\le \frac{1}{4}$ and denoting $\gamma_0 := \min_{x\in\overline X} (\rho^3 e^{\rho(2x_2 + x_3)}) >0$, we have that $|\det(S_1,S_2,S_3)|\ge \frac{\gamma_0}{4}$ on sets of the form $X\cap \{ \rho x_1\in ]\frac{-\pi}{3}, \frac{\pi}{3}[ + m\pi\}$ and $|\det(S_1,S_2,S_4)|\ge \frac{\gamma_0}{4}$ on sets of the form $X\cap \{ \rho x_1\in ]\frac{\pi}{6}, \frac{5\pi}{6}[ + m\pi\}$, where $m$ is a signed integer. Since the previous sets are open and a finite number of them covers $X$ (because $X$ is bounded and $\rho$ is finite), we therefore have fulfilled the desired construction. Upon changing the sign of $S_3$ or $S_4$ on each of these sets if necessary, we can assume that the determinants are all positive.
    
    Let $\{g_i\}_{1\le i\le 4}$ be the traces of the above CGO solutions $(u_{\bm\rho_1}^\Re, u_{\bm\rho_1}^\Im, u_{\bm\rho_2}^\Re, u_{\bm\rho_1}^\Im)$. These illuminations generate solutions that satisfy the desired properties. Any boundary conditions $\tilde g_i$ in an open set sufficiently close to $g_i$ will ensure that the maximum of the determinants stay bounded from below by $c_0>0$:
\begin{align*}
	\max (\det (S_1,S_2,S_3), \det(S_1,S_2,S_4)))(x) \ge c_0 >0, \quad x\in \overline X.
    \end{align*}  
    This concludes the proof of the lemma.
\end{proof}

%\begin{remark}
%    In practice, the illuminations constructed in the above proof are not accessible to us since we do not know the functions $\psi_{\bm\rho}$, which depend on the unknown $\sigma$. However, if $\sigma$ is known at the boundary $\partial X$, one can choose illuminations of the form $\sigma^{-\frac{1}{2}} e^{\bm\rho\cdot x}|_{\partial X}$ and generate solutions that will approximate CGO's all the better that $|\bm\rho|$ will be large. Then by continuity arguments, since these illuminations will become close to the $g_i$'s constructed in the previous proof, they will still ensure a condition of the form 
%    \begin{align*}
%	\max (\det (S_1,S_2,S_3), \det(S_1,S_2,S_4)))(x) \ge c_0 >0, \quad x\in \overline X.
%    \end{align*}
%\end{remark}

%
%%%%%%%%%%%%%%%%%%%%
\subsection{Global reconstruction procedure}

Let us consider $m\ge 3$ solutions of \eqref{eq:conductivity} and assume that there exist an open covering $\O = \{\Omega_i\}_{1\le i\le N}$ with $X\subset\cup_{i=1}^N \Omega_i$, a constant $c_0>0$, and a function $\tau$ such that assumption \eqref{eq:condji} holds. Let us fix $x_0\in\Omega_{i_0}$ such that $\sigma(x_0)$ and $\{S_{\tau(i_0)_i}(x_0) \}_{1\le i\le 3}$ are known. We now assume that there exists an integer $K>0$ and two functions
\begin{align}
    \begin{split}
	Y &: X\ni x \mapsto Y(x) = (y_1(x) = x_0, y_2(x), \dots, y_{K}(x), y_{K+1}(x)=x)\in X^{K+1}, \\ 
	\psi &: X\times [1,K]\ni (x,i) \mapsto \psi(x,i) \in [1,N],
    \end{split}
    \label{eq:psi}
\end{align}
such that for every $x\in X$,
\begin{align}
    [x_0,x] = \cup_{i=1}^K [y_{i}(x),y_{i+1}(x)] \qandq [y_{i},y_{i+1}]\subset \Omega_{\psi(x,i)}, \quad 1\le i\le K.
    \label{eq:condpsi}
\end{align}
An example of such a setting is shown in Figure \ref{fig:setting}. 

\begin{figure}[htpb]
    \begin{center}
	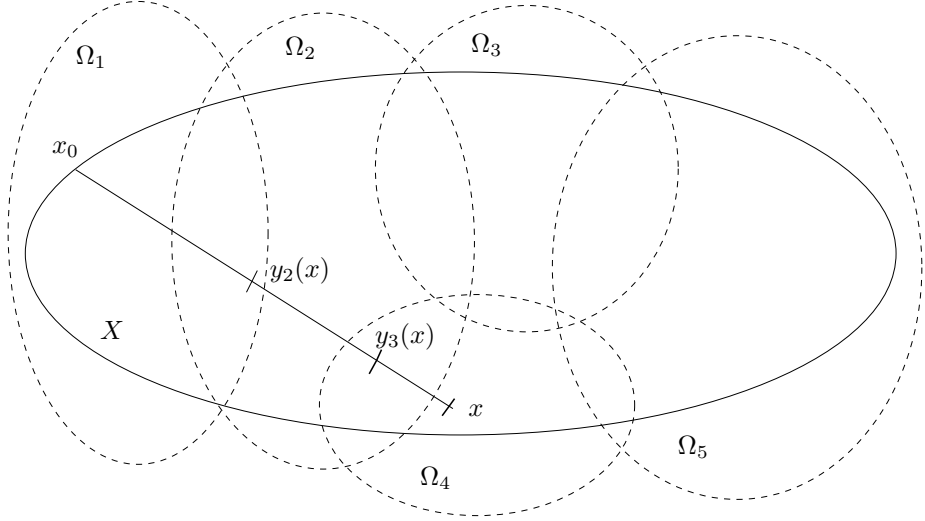
    \end{center}
    \caption{An example of setting for the reconstruction of $\sigma$ at some $x\in X$. One can see that $[x_0,y_2(x)]\subset\Omega_1$, $[y_2(x),y_3(x)]\subset\Omega_2$ and $[y_3(x),x]\subset\Omega_4$. In this case, we have $\psi(x,1) =1$, $\psi(x,2)=2$ and $\psi(x,3)=4$.}
    \label{fig:setting}
\end{figure}

From the $4$-uple $(\O, \tau, Y, \psi)$, we define the reconstruction procedure $\P(\O, \tau, Y, \psi)$ as follows. Let $x\in X$. For every $1\le k\le N$, let $S^{(k)} = [S_{\tau(\psi(x,k))_1}|S_{\tau(\psi(x,k))_2}|S_{\tau(\psi(x,k))_3}]$, $H^{(k)} = \{H_{ij}\}_{i,j\in \tau(\psi(x,k))}$, and $T^{(k)}$ the matrix of Gram-Schmidt coefficients of $H^{(k)}$. For further use, we also define 
\begin{align}
    H^{k-1,k}:= \{ H_{ij} (y_k)\}_{i\in \tau(\psi(x,k-1)), j\in \tau(\psi(x,k))}.
    \label{eq:Htrans}
\end{align}

The reconstruction procedure $\P(\O,\tau,Y,\psi)$ is summarized in algorithm \ref{algo}.
%%--------------------------------------------------------------------------------------------------------------------------------------
\begin{algorithm}
    \caption{Reconstruction procedure $\P(\O,\tau,Y,\psi)$}
\begin{algorithmic} \label{algo}
    \FOR{$x\in X$ and $1\le k\le K$}
    \STATE construct $T^{(k)}$ and the vector fields $V_{lm}^{(k)}$\eqref{eq:Vijs} from the measurements $\{H_{lm}\}_{l,m\in \tau(\psi(x,k))}$, and define $R^{(k)}= S^{(k)} T^{(k)T}$ by orthonormalizing $S^{(k)}$.
    \STATE 
    \IF{$k=1$} 
    \STATE $R^{(1)}(x_0)$ is known from the knowledge of $S^{(1)}(x_0)$.
    \ELSE 
    \STATE construct the value of $R^{(k)}(y_k)$ from the value $R^{(k-1)}(y_k)$ using the following formula
	\begin{align}
		    R^{(k)}(y_k(x)) = R^{(k-1)}(y_k(x))\cdot T^{(k-1)}(y_k(x))\cdot H^{k-1,k}\cdot T^{(k)}(y_k(x)).
		   \label{eq:Rtransfer}
		\end{align}
    \ENDIF
    \STATE For all $y\in [y_k(x), y_{k+1}(x)]$, compute $R^{(k)}(y)$ by integrating \eqref{eq:ODER} along $[y_k(x), y]$.
    \STATE Compute $\sigma(y_{k+1}(x))$ from $\sigma(y_{k}(x))$ and $R^{(k)}(y), y\in[y_k(x),y_{k+1}(x)]$ by integrating equation \eqref{eq:datatoF} along $[y_k(x),y_{k+1}(x)]$.
    \ENDFOR
\end{algorithmic}
\end{algorithm}
%%----------------------------------------------------------------------------------------------------------------------------------------------------------

%\begin{enumerate}
%    \item construct $T^{(k)}$ and the vector fields $V_{lm}^{(k)}$\eqref{eq:Vijs} from the measurements $\{H_{lm}\}_{l,m\in \tau(\psi(x,k))}$, and define $R^{(k)}= S^{(k)} T^{(k)T}$ by orthonormalizing $S^{(k)}$.
%    \item \begin{description}
%	    \item[if $k=1$:] $R^{(1)}(x_0)$ is known from the knowledge of $S^{(1)}(x_0)$.
%	    \item[if $k>1$:] construct the value of $R^{(k)}(y_k)$ from the value $R^{(k-1)}(y_k)$ using the following formula (justified after the algorithm)
%		\begin{align}
%		    R^{(k)}(y_k(x)) = R^{(k-1)}(y_k(x))\cdot T^{(k-1)}(y_k(x))\cdot H^{k-1,k}\cdot T^{(k)}(y_k(x)).
%		    \label{eq:Rtransfer}
%		\end{align}
%	\end{description}
%    \item For all $y\in [y_k(x), y_{k+1}(x)]$, compute $R^{(k)}(y)$ by integrating \eqref{eq:ODER} along $[y_k(x), y]$.
%    \item Compute $\sigma(y_{k+1}(x))$ from $\sigma(y_{k}(x))$ and $R^{(k)}(y), y\in[y_k(x),y_{k+1}(x)]$ by integrating equation \eqref{eq:datatoF} along $[y_k(x),y_{k+1}(x)]$.
%\end{enumerate}

Here and below, the superscript $^{-T}$ denotes the matrix inverse transpose. That the transfer equation \eqref{eq:Rtransfer} holds can be seen from the following facts:
\begin{itemize}
    \item one can compute $S^{(k-1)}(y_k)$ from $R^{(k-1)}(y_k)$ by writing 
	\begin{align*}
	    S^{(k-1)}(y_k) = R^{(k-1)}(y_k)\cdot (T^{(k-1)}(y_k))^{-T}.	    
	\end{align*}
    \item Then, from the equality $(S^{(k-1)}(y_k))^T S^{(k)}(y_k) = H^{k-1,k}$ with $H^{k-1,k}$ defined in \eqref{eq:Htrans}, and the fact that $S^{(k-1)}(y_k)$ is known, one obtains $S^{(k)}(y_k)$ by writing $S^{(k)}(y_k) = (S^{(k-1)}(y_k))^{-T} H^{k-1,k}$.
    \item Finally, one computes $R^{(k)}(y_k(x))$ by formula $R^{(k)}(y_k(x)) = S^{(k)}(y_k(x)) T^{(k)}(y_k(x))$.
\end{itemize} 
Combining the three formulas above gives \eqref{eq:Rtransfer}. 

In order for this procedure to yield globally stable reconstructions, we not only need local stability inside of each $\Omega_i\in\O$ (guaranteed by proposition \ref{prop:stab3d}), but also stability as one passes from $\Omega_{\psi(x,k-1)}$ to $\Omega_{\psi(x,k)}$ using formula \eqref{eq:Rtransfer}. This is done in the following lemma.
\begin{lemma} \label{lem:stab}
    There exist constants $C_2$ and $C_3$ such that for every $x\in X$ and every $2\le k\le K$, we have
    \begin{align}
	\| \bfR^{(k)}(y_k(x)) - \widetilde\bfR^{(k)}(y_k(x)) \| \le C_2 \| \bfR^{(k-1)}(y_k(x)) - \widetilde\bfR^{(k-1)}(y_k(x)) \| + C_3 \|\wtH-\wtH\|_{1,\infty}.
	\label{eq:stab_interstep}
    \end{align}
\end{lemma}

\begin{proof}
    The transfer equation \eqref{eq:Rtransfer} can be summarized as $R^{(k)}(y_k) = R^{(k-1)}(y_k) M^{k-1,k}$, where $M^{k-1,k} := T^{(k-1)}(y_k(x))\cdot H^{k-1,k}\cdot T^{(k)}(y_k(x))$, and similarly for $\wtR^{(k)}(y_k)$. Provided condition \eqref{eq:condji} holds, $T^{(k-1)}$ and $T^{(k)}$ are well-defined and we have estimates of the form 
    \begin{align}
	\|\wtM^{k-1,k}\| \le C \|\wtH\|_{1,\infty} \qandq \|M^{k-1,k}- \wtM^{k-1,k}\| \le C \|H-\wtH\|_{1,\infty},
	\label{eq:Mest}
    \end{align}
    where $\|M\|$ denotes any matrix norm. The lemma is thus proved if we write
    \begin{align*}
	R^{(k)}(y_k) - \wtR^{(k)}(y_k) = (R^{(k-1)}(y_k)- \wtR^{(k-1)}(y_k)) M^{k-1,k} + \wtR^{(k-1)}(y_k) (M^{k-1,k}-\wtM^{k-1,k} ),
    \end{align*}
    and bound $\|R^{(k)}(y_k) - \wtR^{(k)}(y_k)\|$ using triangle inequalities, estimates \eqref{eq:Mest} and the fact that $\|\wtR^{(k-1)}(y_k)\|\le 1$.     
\end{proof}

We now conclude the proof of theorem \ref{thm:stab3d}. 
\begin{proof}[Proof of theorem \ref{thm:stab3d}]
    Let us define the error function
    \begin{align*}
	\varepsilon(x):= \log\sigma(x)-\log\tilde\sigma(x).	
    \end{align*}
    Using equality \eqref{eq:datatoF}, we obtain that 
    \begin{align*}
	\frac{3}{2} \nabla\varepsilon(x) &= \nabla(\log D^{(K)}-\log\wtD^{(K)})  \\
	&\quad + ((V_{ij}^{(K)} +V_{ji}^{(K)})\cdot R_i^{(K)}) R_j^{(K)} - ((\wtV_{ij}^{(K)}+\wtV_{ji}^{(K)}) \cdot \wtR_i^{(K)}) \wtR_j^{(K)},
    \end{align*}
    where $i,j$ run over the set $\tau(\psi(x,K))$ and the superscript $(K)$ denotes the fact that we are at the $K$-th (i.e. last) step of the reconstruction. By writing 
    \begin{align*}
	((V_{ij} +V_{ji})\cdot R_i) R_j &- ((\wtV_{ij} +\wtV_{ji}) \cdot \wtR_i) \wtR_j = (( V_{ij}-\wtV_{ij} + V_{ji}-\wtV_{ji} )\cdot R_i)R_j + \\
	&+ ((\wtV_{ij}-\wtV_{ji})\cdot (R_i-\wtR_i))R_j + ((\wtV_{ij}-\wtV_{ji})\cdot \wtR_i)(R_j-\wtR_j),
    \end{align*}
    we obtain the following bounds% as follows
    \begin{align}
	\frac{3}{2} \|\nabla\varepsilon(x)\| \le C \|H-\wtH \|_{1,\infty} + 2\|\wtH\|_{1,\infty} \|\bfR^{(K)}(x)-\widetilde\bfR^{(K)}(x)\|.
	\label{eq:bound1}
    \end{align}
    We now apply alternatively $K$ times proposition \ref{prop:stab3d} and lemma \ref{lem:stab}. Define $C_4 = \max(C_0,C_2)$ and $C_5 = \max(C_1,C_3)$ where $C_0,C_1$ are given in proposition \ref{prop:stab3d} and $C_2,C_3$ are given in lemma \ref{lem:stab}. We obtain
      \begin{align}
	\|\bfR^{(K)}(x)-\widetilde\bfR^{(K)}(x)\| &\le C_4 \| \bfR^{(K)}(y_K(x)) - \widetilde\bfR^{(K)}(y_K(x))\| + C_5 \|H-\wtH\|_{1,\infty} \nonumber\\
	&\le C_4^2 \| \bfR^{(K-1)}(y_K(x)) - \widetilde\bfR^{(K-1)}(y_K(x)) \| + C_5 (1+C_4) \|H-\wtH\|_{1,\infty} \nonumber\\
	&\le \dots \nonumber\\
	&\le C_4^{2K} \| \bfR^{(1)}(y_1(x)) - \widetilde\bfR^{(1)}(y_1(x)) \| \nonumber \\&\qquad + C_5 (1+C_4\dots+ C_4^{2K-1}) \|H-\wtH\|_{1,\infty} \nonumber\\
	&= C_4^{2K} \| \bfR^{(1)}(x_0) - \widetilde\bfR^{(1)}(x_0) \| + C_5 \frac{C_4^{2K}-1}{C_4-1} \|H-\wtH\|_{1,\infty}. \label{eq:bound2}
    \end{align}
    Using the fact that  $\bfR^{(1)}(x_0)$ and $\widetilde\bfR^{(1)}(x_0)$ are built from $\{S_{\tau(i_0)_l}(x_0), \wtS_{\tau(i_0)_l}(x_0)\}_{1\le l\le 3}$, we obtain an inequality of the form
    \begin{align}
	\| \bfR^{(1)}(x_0) - \widetilde\bfR^{(1)}(x_0) \| \le C_7 \sum_{l=1}^3 \|S_{j_{i_0,l}}(x_0)- \wtS_{j_{i_0,l}}(x_0)\| + C_8 \|H-\wtH\|_{1,\infty}.
	\label{eq:bound3}
    \end{align}
    Summing up the three bounds \eqref{eq:bound1}, \eqref{eq:bound2} and \eqref{eq:bound3}, we obtain that
    \begin{align*}
	\|\nabla\varepsilon(x)\| \le C \sum_{l=1}^3 \|S_{\tau(i_0)_l}(x_0)- \wtS_{\tau(i_0)_l}(x_0)\| +  C' \|H-\wtH\|_{1,\infty}.
    \end{align*}
    Considering that $\varepsilon(x)$ is bounded by $|\varepsilon(x_0)| + \Delta(X) \|\nabla\varepsilon\|_\infty$ for any $x\in X$, inequality \eqref{eq:globstab3d} is proved. 
\end{proof}

\section*{Acknowledgment} The work of GB and FM was partially funded by the NSF under Grant DMS-0804696.

%\bibliography{../../bibliography}
%\bibliographystyle{siam}

\end{document}